\documentclass[11pt,a4paper]{article}
\usepackage{amsmath,amsthm,amsfonts,amssymb,mathrsfs}
\usepackage{graphicx}
\usepackage{multirow}
\usepackage[dvips]{color}
\usepackage[lined, linesnumbered, algoruled, algosection]{algorithm2e}

\graphicspath{{figures/}}
\usepackage[dvipdfm, bookmarksnumbered=true, linkcolor=blue]{hyperref}

\begin{document}
\numberwithin{equation}{section}
\newtheorem{definition}{Definition}[section]
\newtheorem{theorem}{Theorem}[section]
\newtheorem{proposition}{Proposition}[section]
\newtheorem{corollary}{Corollary}[section]
\newtheorem{lemma}{Lemma}[section]
\newtheorem{remark}{Remark}[section]

\title{A Symmetry-based Decomposition Approach to Eigenvalue Problems:
Formulation, Discretization, and Implementation
\thanks{This work was partially supported by
the National Science Foundation of China under Grant 61033009,
the Funds for Creative Research Groups of China under Grant 11021101,
the National Basic Research Program of China under Grants
2011CB309702 and 2011CB309703,
the National High Technology Research and Development Program of China
under Grant 2010AA012303,
and the National Center for Mathematics and Interdisciplinary Sciences, CAS.}}

\author{Jun Fang \thanks{LSEC, Institute of Computational Mathematics and
Scientific/Engineering Computing, Academy of Mathematics and Systems
Science, Chinese Academy of Sciences, Beijing 100190, China and
Graduate University of Chinese Academy of Sciences, Beijing 100190,
China ({\tt fangjun@lsec.cc.ac.cn}).}
\and Xingyu Gao \thanks{HPCC, Institute of Applied Physics and Computational
Mathematics, Beijing 100094, China (\tt{gao\_xingyu@iapcm.ac.cn}).}
\and Aihui Zhou\thanks{LSEC, Institute of Computational Mathematics and
Scientific/Engineering Computing, Academy of Mathematics and Systems Science,
Chinese Academy of Sciences, Beijing 100190, China ({\tt azhou@lsec.cc.ac.cn}).}}

\date{}
\maketitle

{\bf Abstract.} 
In this paper, we propose a decomposition approach for eigenvalue problems
with spatial symmetries, including the formulation, discretization
as well as implementation.
This approach can handle eigenvalue problems with either Abelian or
non-Abelian symmetries, and is friendly for grid-based discretizations
such as finite difference, finite element or finite volume methods.
With the formulation, we divide the original eigenvalue problem into
a set of subproblems and require only a smaller number of eigenpairs
for each subproblem.
We implement the decomposition approach with finite elements and
parallelize our code in two levels.
We show that the decomposition approach can improve the efficiency
and scalability of iterative diagonalization.
In particular, we apply the approach to solving Kohn--Sham equations of
symmetric molecules consisting of hundreds of atoms.
\vskip 0.2cm

{\bf Keywords.} Eigenvalue, Grid-based discretization, Symmetry, Group theory,
Two-level parallelism.

\section{Introduction}
Efficient numerical methods for differential eigenvalue problems
become significant in scientific and engineering computations.
For instance, many properties of molecular systems or solid-state materials
are determined by solving Schr{\"o}dinger-type eigenvalue problems,
such as Hartree--Fock or Kohn--Sham equations \cite{cances03,martin04,young01};
the vibration analysis of complex structures is achieved by solving
the eigenvalue problems derived from the equation of motion
\cite{bennighof04,craig68,hurty60}.
We understand that a lot of eigenpairs have to be computed when
the size of the molecular system is large in electronic structure study,
or the frequency range of interest is increased in structural analysis.
To obtain accurate approximations, we see that a large number of
degrees of freedom should be employed in discretizations.

Since the computational cost grows in proportion to $N_e^2N$,
where $N_e$ is the number of required eigenpairs
and $N$ the number of degrees of freedom,
we should decompose such large-scale eigenvalue problems over domain or
over required eigenpairs.
However, it is not easy to decompose an eigenvalue problem because
the problem has an intrinsic nonlinearity and
is set as a global optimization problem with orthonormal constraints.
We observe that the existing efficient domain decomposition methods
for boundary value problems usually do not work well for eigenvalue problems.

For an eigenvalue problem with symmetries, we are happy to see that
the symmetries may provide a way to do decomposition.
Mathematically, each symmetry corresponds to an operator,
such as a reflection, a rotation or an inversion,
that leaves the object or problem invariant.
Group theory provides a systematic way to exploit symmetries
\cite{bishop93,cornwell97,cotton90,jones60,tinkham64,wigner59}.
Using group representation theory, we may decompose the eigenspace into
some orthogonal subspaces. More precisely, the decomposed subspaces have
distinct symmetries and are orthogonal to each other. However, there are
real difficulties in the implementation of
using symmetries \cite{banjai07,bossavit93}.

We see from quantum physics and quantum chemistry that
people use the so-called symmetry-adapted bases to
approximate eigenfunctions in such orthogonal subspaces.
The symmetry-adapted bases are constructed from specific basis functions
like atomic orbitals, internal coordinates of a molecule, or
orthogonalized plane waves \cite{bishop93,cornwell97,cotton90}.
A case-by-case illustration of the way to construct these bases from
atomic orbitals has been given in \cite{cotton90}, from which we can see that
the construction of symmetry-adapted bases is not an easy task.

We observe that grid-based discretizations,
such as finite difference, finite element and finite volume methods,
are widely used in scientific and engineering computations
\cite{babuska89,babuska91,chatelin83,dai-yang-zhou08,hackbusch92}.
For instance, the finite element method is often used to discretize
eigenvalue problems in structural analysis
\cite{bennighof04,hurty60,zienkiewicz05}.
In the last two decades, grid-based discretization approaches have been
successfully applied to modern electronic structure calculations,
see \cite{beck00,dai11,pask05,tors06} and reference cited therein.
In particular, grid-based discretizations have good locality and
have been proven to be well accommodated to peta-scale computing
by treating extremely large-scale eigenvalue problems arising from
the electron structure calculations \cite{hasegawa11,iwata10}.
Note that grid-based discretizations usually come with a large number of
degrees of freedom. And finite difference methods do not have basis functions
in the classical sense. These facts increase the numerical difficulty to
construct symmetry-adapted bases.

In this paper, we propose a new decomposition approach to
differential eigenvalue problems with symmetries,
which is friendly for grid-based discretizations and does not need the 
explicit construction of symmetry-adapted bases.
We decompose an eigenvalue problem with Abelian or non-Abelian symmetries
into a set of eigenvalue subproblems characterized by
distinct conditions derived from group representation theory.
We use the characteristic conditions directly in grid-based
discretizations to form matrix eigenvalue problems.
Beside the decomposition approach, we provide a construction procedure for
the symmetry-adapted bases.
Then we illustrate the equivalence between our approach
and the approach that constructs symmetry-adapted bases,
by deducing the exact relation between the two discretized problems.

We implement the decomposition approach based on finite element discretizations.
Subproblems corresponding to different irreducible representations can be
solved independently. Accordingly, we parallelize our code in two levels,
including a fundamental level of spatial parallelization and
another level of subproblem distribution.
We apply the approach to solving the Kohn--Sham equation of some cluster
systems with symmetries. Our computations show that the decomposition
approach would be appreciable for large-scale eigenvalue problems.
The implementation techniques can be adapted to finite difference and
finite volume methods, too.

The computational overhead and memory requirement can be reduced by
our decomposition approach.
Required eigenpairs for the original problem are distributed
among subproblems; namely, only a smaller number of eigenpairs
are needed for each subproblem.
And subproblems can be solved in a small subdomain.
Here we give an example to illustrate the effectiveness of
the decomposition approach.
Consider the eigenvalue problem for the Laplacian in domain $(-1,1)^3$
with zero boundary condition, and solve the first 1000 smallest eigenvalues
and associated eigenfunctions.
We decompose the eigenvalue problem into 8 decoupled eigenvalue subproblems
by applying Abelian group $D_{2h}$ which has 8 symmetry operations.
The number of computed eigenpairs for each subproblem is 155, and
the number of degrees of freedom for solving each subproblem, 205,379, is
one eighth of that for the original problem, 1,643,032.
We obtain a speedup of 28.8 by solving 8 subproblems
instead of the original problem.

We should mention that group theory has been introduced to
partial differential equations arising from structural analysis in
\cite{bossavit86, bossavit93}, which mainly focused on boundary value problems
and did not provide any numerical result.
We understand that the design and implementation of decomposition methods
for eigenvalue problems are different from boundary value problems.
We also see that Abelian symmetries have been utilized to simplify the solving
of Kohn--Sham equations in a finite difference code \cite{kronik06}.
However, the implementation in \cite{kronik06} is only applicable to
Abelian groups, in which any two symmetry operations are commutative.
Even in quantum chemistry, most software packages only utilize Abelian groups
\cite{young01}.
In some plane-wave softwares of electronic structure calculations,
symmetries are used to simplify the solving of Kohn--Sham equations by
reducing the number of $k$-points to the irreducible Brillouin zone (IBZ).
For a given $k$-point, they do not classify the eigenstates and thus still
solve the original eigenvalue problem.

The rest of this paper is organized as follows.
In Section~\ref{sec:formulation}, we show the symmetry-based decomposition of
eigenvalue problems, and propose a subproblem formulation proper for
grid-based discretizations.
Then in Section~\ref{sec:discretization}, we give matrix eigenvalue problems
derived from the subproblem formulation
and provide a construction procedure for the symmetry-adapted bases,
from which we deduce the relation of our discretized problems to those
formed by symmetry-adapted bases.
We quantize the decrease in computational cost when using the decomposition
approach in Section~\ref{sec:cost}.
And in Section~\ref{sec:impl} we present some critical implementation issues.
In Section~\ref{sec:numer}, we give numerical examples to
validate our implementation for Abelian and non-Abelian symmetry groups
and show the reduction in computational and communicational overhead;
then we apply the decomposition approach to solving the Kohn--Sham equation
of three symmetric molecular systems with hundreds of atoms.
Finally, we give some concluding remarks.

\section{Decomposition formulation}\label{sec:formulation}
In this section, we recall several basic but useful results of group theory
and propose a symmetry-based decomposition formulation.
The formulation, summarized as Theorem \ref{thm:subproblem} and
Corollary \ref{crl:subproblem1},
can handle eigenvalue problems with Abelian or non-Abelian symmetries.
Some notation and concepts will be given in Appendix A.

\subsection{Representation, basis function, and projection operator}
We start from orthogonal coordinate transformations in
$\mathbb{R}^d~(d=1,2,3)$ such as a rotation, a reflection or an inversion,
that form a finite group $G$ of order $g$.
Let $\Omega\subset\mathbb{R}^d$ be a bounded domain and
$V\subset L^2(\Omega)$ a Hilbert space of functions on $\Omega$
equipped with the $L^2$ scalar product $(\cdot,\cdot)$.
Each $R\in G$ corresponds to an operator $P_R$ on $f\in V$ as
\[ P_R f(Rx) = f(x)\quad\forall x\in\Omega. \]
It is proved that $\{P_R: R\in G\}$ form a group isomorphic to $G$.

A matrix representation of group $G$ means a group of matrices
which is homomorphic to $G$.
Any matrix representation with nonvanishing determinants is equivalent to
a representation by unitary matrices (referred to as unitary representation).
In the following we focus on unitary representations of group $G$.

The great orthogonality theorem
(cf. \cite{cornwell97,jones60,tinkham64,wigner59}) tells that,
all the inequivalent, irreducible, unitary representations
$\{\Gamma^{(\nu)}\}$ of group $G$ satisfy
\begin{equation}
    \label{eq:great_orth}
    \sum_{R\in G} {\Gamma^{(\nu)}(R)}^*_{ml}
    \Gamma^{(\nu')}(R)_{m'l'} = \delta_{\nu\nu'} \delta_{mm'} \delta_{ll'}
    \frac{g}{d_{\nu}}
\end{equation}
for any $l,m\in\{1,2,\ldots,d_{\nu}\}$ and $l',m'\in\{1,2,\ldots,d_{\nu'}\}$,
where $d_{\nu}$ denotes the dimensionality of the $\nu$-th representation
$\Gamma^{(\nu)}$ and ${\Gamma^{(\nu)}(R)}^*_{ml}$ is the complex conjugate
of $\Gamma^{(\nu)}(R)_{ml}$.
The number of all the inequivalent, irreducible, unitary representations
is equal to the number of classes in $G$. We denote this number as $n_c$.

\begin{definition}
    Given $\nu\in\{1,2,\ldots,n_c\}$, non-zero functions
    $\{\phi^{(\nu)}_l:l=1,2,\ldots,d_{\nu}\} \subset V$
    are said to form a basis for $\Gamma^{(\nu)}$ if for any
    $l\in\{1,2,\ldots,{d_{\nu}}\}$
    \begin{equation}
        \label{eq:basis}
        P_R\phi^{(\nu)}_l = \sum_{m=1}^{d_{\nu}}
        \phi^{(\nu)}_{m}\Gamma^{(\nu)}(R)_{ml}\quad\forall R\in G.
    \end{equation}
    Function $\phi^{(\nu)}_l$ is called to belong to the $l$-th column of
    $\Gamma^{(\nu)}$ (or adapt to the $\nu$-$l$ symmetry),
    and $\{\phi^{(\nu)}_{m}:m=1,2,\ldots,d_{\nu},m\ne l\}$ are its partners.
\end{definition}

There holds an orthogonality property for the basis functions
(cf. \cite{tinkham64,wigner59}):
if $\{\phi^{(\nu)}_{l}:l=1,2,\ldots,d_{\nu}\}$ and
$\{\psi^{(\nu')}_{l'}:l'=1,2,\ldots,d_{\nu'}\}$ are basis functions
for irreducible representations $\Gamma^{(\nu)}$ and $\Gamma^{(\nu')}$,
respectively, then
\begin{equation}
    \label{eq:two_orth}
    (\phi_l^{(\nu)}, \psi_{l'}^{(\nu')}) =
    \delta_{\nu\nu'} \delta_{ll'} d_{\nu}^{-1}
    \sum_{m=1}^{d_{\nu}} (\phi_{m}^{(\nu)}, \psi_{m}^{(\nu)})
\end{equation}
holds for any $l\in\{1,2,\ldots,d_{\nu}\}$ and $l'\in\{1,2,\ldots,d_{\nu'}\}$.
This equation implies that, two functions are orthogonal if they belong to
different irreducible representations
or to different columns of the same unitary representation.
And the scalar product of two functions belonging to the same column of
a given unitary representation (or adapting to the same symmetry)
is independent of the column label.

Multiplying equation (\ref{eq:basis}) by ${\Gamma^{(\nu')}(R)}^*_{m'l'}$
and summing over $R$, the great orthogonality theorem (\ref{eq:great_orth})
implies that
\[ \sum_{R\in G} {\Gamma^{(\nu')}(R)}^*_{m'l'} P_R\phi^{(\nu)}_l
= \delta_{\nu\nu'} \delta_{ll'} \frac{g}{d_{\nu}} \phi^{(\nu)}_{m'}
\quad\forall~l',m'\in\{1,2,\ldots,d_{\nu'}\},~l\in\{1,2,\ldots,{d_{\nu}}\}. \]
Define for any $\nu\in\{1,2,\ldots,n_c\}$ and $l,m\in\{1,2,\ldots,d_{\nu}\}$
operator $\mathscr{P}^{(\nu)}_{ml}$ as
\begin{equation}
    \label{eq:Pml_def}
    \mathscr{P}^{(\nu)}_{ml} = \frac{d_{\nu}}{g}
    \sum_{R\in G} {\Gamma^{(\nu)}(R)}^*_{ml} P_R,
\end{equation}
we get
\begin{equation}
    \label{eq:Pml_basic}
    \mathscr{P}^{(\nu)}_{ml} \phi^{(\nu')}_{l'} =
    \delta_{\nu\nu'} \delta_{ll'} \phi^{(\nu)}_{m}
\end{equation}
for any $\nu,\nu'\in\{1,2,\ldots,n_c\}$, $l,m\in\{1,2,\ldots,d_{\nu}\}$,
and $l'\in\{1,2,\ldots,{d_{\nu'}}\}$.

\begin{proposition}
    \label{prop:basis}
    Given $\nu\in\{1,2,\ldots,n_c\}$ and $k\in\{1,2,\ldots,d_{\nu}\}$.
    If $v\in V$ satisfies $\mathscr{P}^{(\nu)}_{kk} v\ne 0$,
    then $\{\mathscr{P}^{(\nu)}_{lk} v:l=1,2,\ldots,d_{\nu}\}$
    form a basis for $\Gamma^{(\nu)}$, i.e.,
    $\{\mathscr{P}^{(\nu)}_{lk} v:l=1,2,\ldots,d_{\nu}\}$
    are non-zero functions, and for any $l\in\{1,2,\ldots,d_{\nu}\}$
    \[ P_R \left(\mathscr{P}^{(\nu)}_{lk}v\right) = \sum_{m=1}^{d_{\nu}}
    \left(\mathscr{P}^{(\nu)}_{mk}v\right) \Gamma^{(\nu)}(R)_{ml}
    \quad\forall R\in G. \]
\end{proposition}
\begin{proof}
    For any $l\in\{1,2,\ldots,d_{\nu}\}$, we obtain from (\ref{eq:Pml_def})
    that
    \[ P_R \left(\mathscr{P}^{(\nu)}_{lk}v\right) =
    \frac{d_{\nu}}{g} \sum_{S\in G}{\Gamma^{(\nu)}(S)}^*_{lk} P_{RS}~v
    = \frac{d_{\nu}}{g}
    \sum_{S'\in G} {\Gamma^{(\nu)}(R^{-1}S')}^*_{lk} P_{S'}v
    \quad\forall R\in G, \]
    where $P_RP_S=P_{RS}$ because $\{P_R: R\in G\}$ form a group isomorphic
    to $G$.

    Since $\Gamma^{(\nu)}$ is a unitary representation of $G$, we have
    \[ P_R \left(\mathscr{P}^{(\nu)}_{lk}v\right) =
    \frac{d_{\nu}}{g}\sum_{S'\in G} \left( \sum_{m=1}^{d_{\nu}}
    \Gamma^{(\nu)}(R)_{ml} {\Gamma^{(\nu)}(S')}^*_{mk} \right) P_{S'}v, \]
    or
    \[ P_R \left(\mathscr{P}^{(\nu)}_{lk}v\right) =
    \sum_{m=1}^{d_{\nu}} \left(\mathscr{P}^{(\nu)}_{mk}v\right)
    \Gamma^{(\nu)}(R)_{ml}\quad\forall R\in G. \]

    Recall the way to achieve (\ref{eq:Pml_basic}), we see from the above
    equation and the great orthogonality theorem that
    \[ \mathscr{P}^{(\nu)}_{kk}v =
    \mathscr{P}^{(\nu)}_{kl} \left(\mathscr{P}^{(\nu)}_{lk}v\right)
    \quad\forall~l\in\{1,2,\ldots,d_{\nu}\}. \]
    So $\mathscr{P}^{(\nu)}_{kk} v\ne 0$ indicates
    $\mathscr{P}^{(\nu)}_{lk} v\ne 0$ for all $l\in\{1,2,\ldots,d_{\nu}\}$.
    This completes the proof.
\qquad\end{proof}

If we set $\nu'=\nu$, $l'=l$ and $m=l$ in (\ref{eq:Pml_basic}), then we have
for any $\nu\in\{1,2,\ldots,n_c\}$
\begin{equation}
    \label{eq:Pml_special}
    \mathscr{P}^{(\nu)}_{ll} \phi^{(\nu)}_l
    = \phi^{(\nu)}_l\quad\forall~l\in\{1,2,\ldots,d_{\nu}\}.
\end{equation}
Proposition \ref{prop:basis} implies that (\ref{eq:Pml_special}) serves to
characterize the labels of any basis function:
\begin{corollary}
    \label{crl:characterize}
    Given $\nu\in\{1,2,\ldots,n_c\}$ and $l\in\{1,2,\ldots,d_{\nu}\}$.
    Non-zero function $v\in V$ belongs to the $l$-th column of $\Gamma^{(\nu)}$
    (or adapts to the $\nu$-$l$ symmetry) if and only if
    \[ \mathscr{P}^{(\nu)}_{ll} v = v. \]
\end{corollary}

We will use the following properties of operator $\mathscr{P}^{(\nu)}_{ml}$,
whose proof is given in Appendix B.
\begin{proposition}
    \label{prop:Pml_property}
    Let $\nu,\nu'\in\{1,2,\ldots,n_c\}$, $l,m\in\{1,2,\ldots,d_{\nu}\}$,
    and $l',m'\in\{1,2,\ldots,d_{\nu'}\}$.

    (a) The adjoint of operator $\mathscr{P}^{(\nu)}_{ml}$ satisfies
    \[ {\mathscr{P}^{(\nu)}_{ml}}^* = \mathscr{P}^{(\nu)}_{lm}. \]

    (b) The multiplication of two operators $\mathscr{P}^{(\nu)}_{ml}$
    and $\mathscr{P}^{(\nu')}_{m'l'}$ satisfies
    \[ \mathscr{P}^{(\nu)}_{ml} \mathscr{P}^{(\nu')}_{m'l'} =
    \delta_{\nu\nu'}\delta_{lm'} \mathscr{P}^{(\nu)}_{ml'}. \]
\end{proposition}

\subsection{Subproblems}
We see from Corollary \ref{crl:characterize} and
the linearity of operator $\mathscr{P}^{(\nu)}_{ll}$ that,
all functions in $V$ belonging to the $l$-th column of $\Gamma^{(\nu)}$
(or adapting to the $\nu$-$l$ symmetry) form a subspace of $V$.
We denote this subspace by $V_l^{(\nu)}$.

There holds a decomposition theorem for any function in $V$
(cf. \cite{tinkham64,wigner59}):
any $f\in V$ can be decomposed into a sum of the form
\begin{equation}
    \label{eq:f_decomp}
    f = \sum_{\nu=1}^{n_c}\sum_{l=1}^{d_{\nu}}f_l^{(\nu)},
\end{equation}
where $f_l^{(\nu)}\in V_l^{(\nu)}$.
We see from (\ref{eq:Pml_basic}) and (\ref{eq:f_decomp}) that
$\mathscr{P}^{(\nu)}_{ll}:V\rightarrow V_l^{(\nu)}$ is a projection operator.
Equation (\ref{eq:f_decomp}) implies
\[ V = \sum_{\nu=1}^{n_c}\sum_{l=1}^{d_{\nu}} V_l^{(\nu)}, \]
which indeed is a direct sum
\begin{equation}
    \label{eq:V_decomp}
    V = \bigoplus_{\nu=1}^{n_c}\bigoplus_{l=1}^{d_{\nu}} V_l^{(\nu)}
\end{equation}
due to (\ref{eq:two_orth}).

Now we turn to study the symmetry-based decomposition for eigenvalue problems.
Consider eigenvalue problems of the form
\begin{equation}
    \label{eq:general}
    Lu = \lambda u \quad\mbox{in}~\Omega
\end{equation}
subject to some boundary condition, where $L$ is an Hermitian operator.
Group $G$ is said to be a symmetry group associated with eigenvalue problem
(\ref{eq:general}) if
\[ R\Omega = \Omega,~ P_R L = L P_R \quad\forall R\in G, \]
and the subjected boundary condition is also invariant under $\{P_R\}$.
Then any $R\in G$ is called a symmetry operation for problem
(\ref{eq:general}).
For simplicity, we take zero boundary condition as an example and
discuss the decomposition of eigenvalue problem
\begin{equation}
    \label{eq:ori-problem}
    \left\{
    \begin{array}{rcll}
        Lu & = &\lambda u &\quad\mbox{in}~\Omega,\\
        u & = &0 &\quad\mbox{on}~\partial\Omega.
    \end{array}
    \right.
\end{equation}

Since $P_R$ and $L$ are commutative for any $R$ in $G$, we have:
\begin{proposition}
    \label{prop:Lv}
    If $v\in V_l^{(\nu)}$, then $Lv\in V_l^{(\nu)}$, where
    $\nu\in\{1,2,\ldots,n_c\}$ and $l\in\{1,2,\ldots,d_{\nu}\}$.
    In other words, $V_l^{(\nu)}$ is an invariant subspace of operator $L$.
\end{proposition}

The direct sum decomposition of space $V$ and Proposition \ref{prop:Lv}
indicate a decomposition of the eigenvalue problem.
\begin{theorem}
    \label{thm:subproblem}
    Suppose finite group $G=\{R\}$ is a symmetry group associated with
    eigenvalue problem (\ref{eq:ori-problem}).
    Denote all the inequivalent, irreducible, unitary representations of $G$
    as $\{\Gamma^{(\nu)}:\nu=1,2,\ldots,n_c\}$.
    Then the eigenvalue problem can be decomposed into
    $\sum_{\nu=1}^{n_c}d_{\nu}$ subproblems.
    For any $\nu\in\{1,2,\ldots,n_c\}$,
    the corresponding $d_{\nu}$ subproblems are
    \begin{equation}
        \label{eq:subproblem}
        \left\{
        \begin{array}{rcll}
            Lu_l^{(\nu)} & = &\lambda^{(\nu)} u_l^{(\nu)}
            &\quad\mbox{in}~\Omega,\\
            u_l^{(\nu)} & = &0 &\quad\mbox{on}~\partial\Omega,
            \qquad l=1,2,\ldots,d_{\nu},\\
            u_l^{(\nu)} &=& \mathscr{P}^{(\nu)}_{lk}u_k^{(\nu)}
            &\quad\mbox{in}~\Omega,
        \end{array}
        \right.
    \end{equation}
    where $k$ is any chosen number in $\{1,2,\ldots,d_{\nu}\}$.
\end{theorem}
\begin{proof}
    We see from (\ref{eq:V_decomp}) and Proposition \ref{prop:Lv} that,
    other than solving the eigenvalue problem in $V$,
    we can solve the problem in each subspace $V_l^{(\nu)}$ independently.
    More precisely, we can decompose the original eigenvalue problem
    (\ref{eq:ori-problem}) into $\sum_{\nu=1}^{n_c} d_{\nu}$ subproblems;
    for any $\nu\in\{1,2,\ldots,n_c\}$, the $d_{\nu}$ subproblems are
    as follows
    \begin{equation}
        \label{eq:subproblem0}
        \left\{
        \begin{array}{rcll}
            Lu_l^{(\nu)}  &=& \lambda_l^{(\nu)} u_l^{(\nu)}
            &\quad\mbox{in}~\Omega,\\
            u_l^{(\nu)} & = &0 &\quad\mbox{on}~\partial\Omega,
            \qquad l=1,2,\ldots,d_{\nu},\\
            \mathscr{P}^{(\nu)}_{ll}u_l^{(\nu)} &=& u_l^{(\nu)}
            &\quad\mbox{in}~\Omega,
        \end{array}
        \right.
    \end{equation}
    where the third equation characterizes $u_l^{(\nu)}\in V_l^{(\nu)}$,
    as indicated in Corollary \ref{crl:characterize}.

    Given any $\nu\in\{1,2,\ldots,n_c\}$, we consider the $d_{\nu}$
    subproblems (\ref{eq:subproblem0}).
    We shall prove that, for any $k,l\in\{1,2,\ldots,d_{\nu}\}$,
    if $v$ and $w$ are two orthogonal eigenfunctions corresponding to
    some eigenvalue of the $k$-th subproblem, then
    $\mathscr{P}^{(\nu)}_{lk}v$ and $\mathscr{P}^{(\nu)}_{lk}w$
    are eigenfunctions of the $l$-th subproblem with the same eigenvalue,
    and are also orthogonal.

    Combining (\ref{eq:Pml_basic}) and the fact that $P_R$ and $L$ are
    commutative for each $R$, we obtain that
    $\mathscr{P}^{(\nu)}_{lk}v$ is an eigenfunction of the $l$-th subproblem
    which corresponds to the same eigenvalue as the one for $v$.
    It remains to prove the orthogonality of
    $\mathscr{P}^{(\nu)}_{lk}v$ and $\mathscr{P}^{(\nu)}_{lk}w$.
    Proposition \ref{prop:Pml_property} indicates that
    the scalar product of any two functions in $V_k^{(\nu)}$
    is invariant after operating on them with $\mathscr{P}^{(\nu)}_{lk}$.
    Indeed, we have for any $v,w\in V_k^{(\nu)}$ that
    \[ (\mathscr{P}^{(\nu)}_{lk}v, \mathscr{P}^{(\nu)}_{lk}w)
    = (\mathscr{P}^{(\nu)*}_{lk}\mathscr{P}^{(\nu)}_{lk}v, w)
    = (\mathscr{P}^{(\nu)}_{kl}\mathscr{P}^{(\nu)}_{lk}v, w)
    = (\mathscr{P}^{(\nu)}_{kk}v, w),\]
    which together with  Corollary \ref{crl:characterize} leads to
    \[ (\mathscr{P}^{(\nu)}_{lk}v, \mathscr{P}^{(\nu)}_{lk}w) = (v, w).\]
    Thus $\mathscr{P}^{(\nu)}_{lk}v$ and $\mathscr{P}^{(\nu)}_{lk}w$
    are orthogonal when $v$ and $w$ are.

    Since $L$ is Hermitian, we see that for the $d_{\nu}$ subproblems
    (\ref{eq:subproblem0}),
    eigenvalues of the $l$-th subproblem are the same as those of the $k$-th
    one, and eigenfunctions of the $l$-th subproblem
    can be chosen as $\{\mathscr{P}^{(\nu)}_{lk}v\}$, where $\{v\}$ are
    eigenfunctions of the $k$-th subproblem and $k$ is any chosen number
    in $\{1,2,\ldots,d_{\nu}\}$.

    Therefore, the original eigenvalue problem (\ref{eq:ori-problem})
    is decomposed into $\sum_{\nu=1}^{n_c} d_{\nu}$ subproblems, and
    for any $\nu\in\{1,2,\ldots,n_c\}$ the corresponding $d_{\nu}$
    subproblems can be given as (\ref{eq:subproblem}).
    This completes the proof.
\qquad\end{proof}

The third equation of the $d_{\nu}$ subproblems in (\ref{eq:subproblem})
are
\[ u_k^{(\nu)} = \mathscr{P}^{(\nu)}_{kk}u_k^{(\nu)}, \]
\[ u_l^{(\nu)} = \mathscr{P}^{(\nu)}_{lk}u_k^{(\nu)}
\quad\forall~l=1,2,\ldots,d_{\nu},~l\ne k. \]
We see from Proposition \ref{prop:basis} that
$\{u_l^{(\nu)}:l=1,2,\ldots,d_{\nu}\}$ form a basis for
$\Gamma^{(\nu)}$. Namely, for any $l\in\{1,2,\ldots,d_{\nu}\}$
\[ P_Ru_l^{(\nu)} = \sum_{m=1}^{d_{\nu}}
u_{m}^{(\nu)}~\Gamma^{(\nu)}(R)_{ml}\quad\forall R\in G, \]
i.e.,
\[ u_l^{(\nu)}(Rx) = \sum_{m=1}^{d_{\nu}}
{\Gamma^{(\nu)}(R)}^*_{lm} u_{m}^{(\nu)}(x) \quad\forall R\in G. \]
\begin{corollary}
    \label{crl:subproblem1}
    Under the same condition as in Theorem \ref{thm:subproblem},
    eigenvalue problem (\ref{eq:ori-problem}) can be decomposed into
    $\sum_{\nu=1}^{n_c}d_{\nu}$ subproblems.
    For any $\nu\in\{1,2,\ldots,n_c\}$, the corresponding $d_{\nu}$
    subproblems can be given as follows
    \begin{equation}
        \label{eq:subproblem1}
        \qquad\left\{
        \begin{array}{rcll}
            Lu_l^{(\nu)} & = &\lambda^{(\nu)} u_l^{(\nu)}
            &~~\mbox{in}~\Omega,\\
            u_l^{(\nu)} & = &0 &~~\mbox{on}~\partial\Omega,
            \qquad\qquad l=1,2,\ldots,d_{\nu}.\\
            u_l^{(\nu)}(Rx) &=& {\displaystyle\sum_{m=1}^{d_{\nu}} }
            {\Gamma^{(\nu)}(R)}^*_{lm} u_{m}^{(\nu)}(x)
            &~~\mbox{in}~\Omega,~\forall R\in G,
        \end{array}
        \right.
    \end{equation}
\end{corollary}

The third equations in (\ref{eq:subproblem}) and (\ref{eq:subproblem1})
describe symmetry properties of eigenfunctions over domain $\Omega$.
The original eigenvalue problem can be decomposed into subproblems
just because eigenfunctions of subproblems satisfy distinct equations.
In the following text, we call these equations as
symmetry characteristics.

Denote by $\Omega_0$ the smallest subdomain which produces $\Omega$
by applying all symmetry operations $\{R\in G\}$, namely,
${\bar \Omega}=\cup_{R\in G}\overline{R\Omega_0}$, and
${R_1\Omega_0}\cap{R_2\Omega_0}=\varnothing$ for any $R_1,R_2\in G$
satisfying $R_1\ne R_2$.
We call $\Omega_0$ the irreducible subdomain and
the associated volume is $g$ times smaller than that of $\Omega$.
The symmetry characteristic equation in (\ref{eq:subproblem1}) tells that
for any $l\in\{1,2,\ldots,d_{\nu}\}$,
$u_l^{(\nu)}$ over $\Omega$ is determined by the values of functions
$\{u_1^{(\nu)},\ldots,u_{d_{\nu}}^{(\nu)}\}$ over $\Omega_0$.
So each subproblem can be solved over $\Omega_0$.

\begin{remark}
    \label{rmk:characteristic}
    A decomposition formulation has been shown in \cite{bossavit93} for
    boundary value problems with spatial symmetries.
    Each decomposed problem is characterized by a ``boundary condition'' on
    $\Sigma_g$ \footnote{In \cite{bossavit93}, $\Sigma_g$ is the ``internal''
    boundary $\partial\Omega_0\setminus\partial\Omega$ of $\Omega_0$, and
    irreducible subdomain $\Omega_0$ is called symmetry cell.},
    which is in fact a restriction of the symmetry characteristic on
    boundary $\Sigma_g$.
    Indeed, symmetry characteristics over $\Omega$ should not be replaced by
    the restriction on the internal boundary.
    In some cases, it is true that boundary conditions such as
    Dirichlet or Neumann type can be deduced, while in the deduction
    of Neumann boundary conditions one has to use the symmetry characteristic
    near the internal boundary, not only on the boundary.
    In some other cases, symmetry characteristics may not produce proper
    boundary conditions.
\end{remark}

\subsection{An example}
We take the Laplacian in square $(-1,1)^2$ as an example to illustrate
the subproblem formulation in Corollary \ref{crl:subproblem1}.
Namely, we consider the decomposition of the following eigenvalue problem
\begin{equation}
    \label{eq:eg}
    \left\{
    \begin{array}{rcll}
        -\Delta u & = &\lambda u &\quad\mbox{in}~\Omega=(-1,1)^2,\\
        u & = &0 &\quad\mbox{on}~\partial\Omega.
    \end{array}
    \right.
\end{equation}
Note that $G=\{E,\sigma_x,\sigma_y,I\}$ is a symmetry group associated with
(\ref{eq:eg}), where $E$ represents the identity operation,
$\sigma_x$ a reflection about $x$-axis,
$\sigma_y$ a reflection about $y$-axis, and $I$ the inversion operation.
We see that $G$ is an Abelian group of order 4, and has 4 one-dimensional
irreducible representations as shown in Table \ref{table:eg_IR}.

\begin{table}[htbp]
    \caption{Representation matrices of example group $G$.}
    \begin{center}\footnotesize
    \begin{tabular}{|c|cccc|}
        \hline
        $G$ & $R_1=E$ & $R_2=\sigma_x$ & $R_3=\sigma_y$ & $R_4=I$ \\
        \hline
        $\Gamma^{(1)}$ & 1 & 1 & 1 & 1 \\
        $\Gamma^{(2)}$ & 1 & 1 &-1 &-1 \\
        $\Gamma^{(3)}$ & 1 &-1 &-1 & 1 \\
        $\Gamma^{(4)}$ & 1 &-1 & 1 &-1 \\ \hline
    \end{tabular}
    \end{center}
    \label{table:eg_IR}
\end{table}

According to Theorem \ref{thm:subproblem} and Corollary \ref{crl:subproblem1},
eigenvalue problem (\ref{eq:eg}) can be decomposed into 4 subproblems
(due to $\sum_{\nu=1}^{n_c}d_{\nu}=4$).
And the symmetry characteristic conditions, the third equation in
(\ref{eq:subproblem1}), for the 4 subproblems are
\begin{eqnarray}
    \{u^{(1)}(R_1x),u^{(1)}(R_2x),u^{(1)}(R_3x),u^{(1)}(R_4x)\}
    &=& \{1, 1, 1, 1\} ~u^{(1)}(x), \label{eq:1} \\
    \{u^{(2)}(R_1x),u^{(2)}(R_2x),u^{(2)}(R_3x),u^{(2)}(R_4x)\}
    &=& \{1, 1,-1,-1\} ~u^{(2)}(x), \label{eq:2} \\
    \{u^{(3)}(R_1x),u^{(3)}(R_2x),u^{(3)}(R_3x),u^{(3)}(R_4x)\}
    &=& \{1,-1,-1, 1\} ~u^{(3)}(x), \label{eq:3} \\
    \{u^{(4)}(R_1x),u^{(4)}(R_2x),u^{(4)}(R_3x),u^{(4)}(R_4x)\}
    &=& \{1,-1, 1,-1\} ~u^{(4)}(x), \label{eq:4}
\end{eqnarray}
where $x\in\Omega$ is an arbitrary point and
subscripts of $\{u_1^{(\nu)}:\nu=1,2,3,4\}$ are omitted.

In Figure~\ref{fig:eg}, we illustrate four eigenfunctions of (\ref{eq:eg})
belonging to different subproblems.
We see that $u_2$ and $u_3$ are degenerate eigenfunctions
corresponding to $\lambda=\frac{5}{4}\pi^2$ with double degeneracy.
In other words, a doubly-degenerate eigenvalue of the original problem
becomes nondegenerate for subproblems.
This implies a relation between symmetry and degeneracy
\cite{kuttler84,neuberger06,trefethen06}.
Moreover, the first subproblem does not have this eigenvalue, which shows
that the decomposition approach has improved the spectral separation.

\begin{figure}[htpb]
    \begin{center}
        \includegraphics[width=0.7\textwidth]{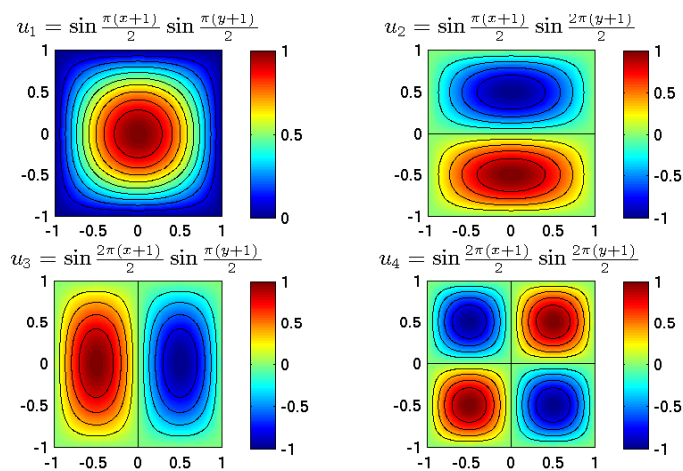}
    \end{center}
    \caption{Four eigenfunctions of problem (\ref{eq:eg}):
    $u_1$ keeps invariant under $\{E,\sigma_x,\sigma_y,I\}$
    and satisfies equation (\ref{eq:1}),
    and $u_2$, $u_3$ and $u_4$ satisfy (\ref{eq:4}), (\ref{eq:2}) and
    (\ref{eq:3}), respectively.}
    \label{fig:eg}
\end{figure}

Under the assumption that all symmetries of the eigenvalue problem
are included in group $G$ and no accidental degeneracy occurs,
the eigenvalue degeneracy is determined by the dimensionalities
of irreducible representations of $G$ \cite{tinkham64,wigner59}.
For example, in cubic crystals
\footnote{Cubic crystals are crystals where the unit cell is a cube.
All irreducible representations of the associated symmetry group are
one-, two-, or three-dimensional.}
all eigenstates have degeneracy 1, 2, or 3 \cite{martin04}.
According to Theorem \ref{thm:subproblem} or Corollary \ref{crl:subproblem1},
eigenvalues of each subproblem should be nondegenerate.
In practice, we usually use part of symmetry operations.
Thus subproblems will probably still have degenerate eigenvalues.
However, it is possible to improve the spectral separation,
especially when we exploit as many symmetries as possible.
This would benefit the convergence of iterative diagonalization.

Formulation (\ref{eq:subproblem1}) makes a straightforward implementation
for grid-based discretizations.
We shall discuss the way to solve the subproblems in the next section.

\section{Discretization}\label{sec:discretization}
In this section, we study the discretized eigenvalue problems for
subproblems (\ref{eq:subproblem}) and (\ref{eq:subproblem1}).
First we deduce our discretized systems when grid-based discretizations
are employed. Then we provide a construction procedure for
the symmetry-adapted bases, based on which we illustrate the relation of
our discretized systems to those formed by symmetry-adapted bases.

Note that the $d_{\nu}$ subproblems associated with different $\nu$ values
are independent and have the same formulation.
So we take one $\nu\in\{1,2,\ldots,n_c\}$ and
discuss the corresponding $d_{\nu}$ subproblems.

\subsection{Our discretized system}\label{sec:discretized_sys}
Suppose $\Omega$ is discretized by a symmetrical grid with respect to group
$G$, and $N$ is the number of degrees of freedom.
For simplicity we assume that no degree of freedom lies on symmetry elements
\footnote{Symmetry element of operation $R$ is a point of reference
about which $R$ is carried out, such as a point to do inversion,
a rotation axis, or a reflection plane.
Symmetry element is invariant under the associated symmetry operation.}.

We determine a smallest set of degrees of freedom that could produce
all $N$ ones by applying symmetry operations $\{R\in G\}$.
It is clear that the number of degrees of freedom in this smallest set
satisfies $N_0=\frac{1}{g}N$. We denote the set as
\[ \{x_j:j=1,2,\ldots,N_0\}, \]
then all degrees of freedom can be given by
\[ \{R(j):j=1,2,\ldots,N_0,~R\in G\}, \]
where $R(j)\equiv Rx_j~(j=1,2,\ldots,N_0)$.

The symmetry characteristic equation in (\ref{eq:subproblem1}) tells that
for any $l\in\{1,2,\ldots,d_{\nu}\}$, the values of $u_l^{(\nu)}$ on
all degrees of freedom $\{R(j):j=1,2,\ldots,N_0,~R\in G\}$ are determined by
the values of $\{u_1^{(\nu)},\ldots,u_{d_{\nu}}^{(\nu)}\}$ on
$\{j:j=1,2,\ldots,N_0\}$.
Thus, the size of discretized eigenvalue problem for (\ref{eq:subproblem1})
is $d_{\nu}N_0$.

If the given irreducible representation $\Gamma^{(\nu)}$ is one-dimensional,
then (\ref{eq:subproblem1}) gives
\begin{equation}
    \label{eq:subproblem11}
    \left\{
    \begin{array}{rcll}
        Lu^{(\nu)} & = &\lambda^{(\nu)} u^{(\nu)}
        &\quad\mbox{in}~\Omega,\\
        u^{(\nu)} & = &0 &\quad\mbox{on}~\partial\Omega, \\
        u^{(\nu)}(Rx) &=&
        {\Gamma^{(\nu)}(R)}^*~u^{(\nu)}(x)
        &\quad\mbox{in}~\Omega,~\forall R\in G,
    \end{array}
    \right.
\end{equation}
where we omit subscripts of ${\Gamma^{(\nu)}(R)}^*_{11}$ and $u_1^{(\nu)}$.

Suppose the discretized system for eigenvalue problem (\ref{eq:subproblem11})
is
\[ \sum_{j=1}^{N_0}\sum_{R\in G} a_{\scriptscriptstyle{i,R(j)}}
u_{\scriptscriptstyle{R(j)}} = \lambda \sum_{j=1}^{N_0}\sum_{R\in G}
b_{\scriptscriptstyle{i,R(j)}} u_{\scriptscriptstyle{R(j)}},
\quad i=1,2,\ldots,N_0,\]
where $u_{\scriptscriptstyle{R(j)}}$ is the unknown associated with $Rx_j$
and $\{a_{\scriptscriptstyle{i,R(j)}}, b_{\scriptscriptstyle{i,R(j)}}\}$
represent the discretization coefficients.
For instance, in finite element discretizations,
$a_{\scriptscriptstyle{i,R(j)}}$ and $b_{\scriptscriptstyle{i,R(j)}}$
are entries of the stiffness and mass matrices, respectively.
Note that for any $i\in\{1,2,\ldots,N_0\}$,
although the discretization equation seems to involve all $N$ degrees of
freedom $\{R(j):j=1,2,\ldots,N_0,~R\in G\}$,
in fact only part of coefficients
$\{a_{\scriptscriptstyle{i,R(j)}},~b_{\scriptscriptstyle{i,R(j)}}:
j=1,2,\ldots,N_0,~R\in G\}$ are non-zero.
An extreme example is that in finite difference discretizations
$b_{\scriptscriptstyle{i,R(j)}}=\delta_{\scriptscriptstyle{i,R(j)}}~
(j=1,2,\ldots,N_0,~R\in G)$ for any $i\in\{1,2,\ldots,N_0\}$.

We know from the symmetry characteristic equation that the discretized system
is then reduced to
\[ \sum_{j=1}^{N_0}\sum_{R\in G}
{\Gamma^{(\nu)}(R)}^*~a_{\scriptscriptstyle{i,R(j)}}
~u_{\scriptscriptstyle{j}} = \lambda \sum_{j=1}^{N_0}\sum_{R\in G}
{\Gamma^{(\nu)}(R)}^*~b_{\scriptscriptstyle{i,R(j)}}
~u_{\scriptscriptstyle{j}},\quad i=1,2,\ldots,N_0.\]
Denote the solution vector as
\[ {\bf u} = (u_1,u_2,\ldots,u_{N_0})^{\mathsf T}, \]
we may rewrite the discretized system as a matrix form
\[ A{\bf u} = \lambda B{\bf u}, \]
where
\begin{equation}
    \label{eq:abelAB}
    \begin{array}{rcll}
        A &=& (A_{ij})_{N_0\times N_0},
        &~A_{ij} = {\displaystyle \sum_{R\in G}}
        {\Gamma^{(\nu)}(R)}^*~a_{\scriptscriptstyle{i,R(j)}},\vspace{0.1cm}\\
        B &=& (B_{ij})_{N_0\times N_0},
        &~B_{ij} = {\displaystyle\sum_{R\in G}}
        {\Gamma^{(\nu)}(R)}^*~b_{\scriptscriptstyle{i,R(j)}}.
    \end{array}
\end{equation}

In the case of higher-dimensional irreducible representations,
the $d_{\nu}$ subproblems in (\ref{eq:subproblem1}) are coupled through
symmetry characteristics.
Taking $d_{\nu}=2$ as an example, we assemble subproblems for $u_1^{(\nu)}$
and $u_2^{(\nu)}$ in (\ref{eq:subproblem1}) to solve eigenvalue problem
{\small
\begin{equation}
    \label{eq:subproblem12}
    \quad\left\{
    \begin{array}{rcll}
        \left[\begin{array}{c} Lu_1^{(\nu)} \vspace{0.1cm}\\ Lu_2^{(\nu)}
        \end{array}\right]
        &=& \lambda^{(\nu)}
        \left[\begin{array}{c}  u_1^{(\nu)} \vspace{0.1cm}\\  u_2^{(\nu)}
        \end{array}\right]
        &\mbox{in}~\Omega, \vspace{0.2cm} \\
        \left[\begin{array}{c}  u_1^{(\nu)} \vspace{0.1cm}\\  u_2^{(\nu)}
        \end{array}\right] (Rx) &=&
        \left[\begin{array}{cc}
            {\Gamma^{(\nu)}(R)}^*_{11} & {\Gamma^{(\nu)}(R)}^*_{12}
            \vspace{0.1cm}\\
            {\Gamma^{(\nu)}(R)}^*_{21} & {\Gamma^{(\nu)}(R)}^*_{22} \\
        \end{array}\right]
        \left[\begin{array}{c}  u_1^{(\nu)} \vspace{0.1cm}\\  u_2^{(\nu)}
        \end{array}\right] (x)
        &\mbox{in}~\Omega,~\forall R\in G, \vspace{0.2cm} \\
        \left[\begin{array}{c}  u_1^{(\nu)} \vspace{0.1cm}\\  u_2^{(\nu)}
        \end{array}\right]
        &=& \left[\begin{array}{c}0 \vspace{0.1cm}\\0 \end{array}\right]
        &\mbox{on}~\partial\Omega.
    \end{array}
    \right.
\end{equation}
}

Suppose the discretized system associated with (\ref{eq:subproblem12}) is
\[\left\{
\begin{array}{c}
    {\displaystyle\sum_{j=1}^{N_0}\sum_{R\in G}}
    ~a_{\scriptscriptstyle{i,R(j)}}~u_{\scriptscriptstyle{1,R(j)}} = \lambda
    ~{\displaystyle\sum_{j=1}^{N_0}\sum_{R\in G}}
    ~b_{\scriptscriptstyle{i,R(j)}}~u_{\scriptscriptstyle{1,R(j)}},
    \quad i=1,2,\ldots,N_0,
    \vspace{0.1cm} \\
    {\displaystyle\sum_{j=1}^{N_0}\sum_{R\in G}}
    ~a_{\scriptscriptstyle{i,R(j)}}~u_{\scriptscriptstyle{2,R(j)}} = \lambda
    ~{\displaystyle\sum_{j=1}^{N_0}\sum_{R\in G}}
    ~b_{\scriptscriptstyle{i,R(j)}}~u_{\scriptscriptstyle{2,R(j)}},
    \quad i=1,2,\ldots,N_0,
\end{array}
\right.
\]
where $u_{\scriptscriptstyle{1,R(j)}}$ and $u_{\scriptscriptstyle{2,R(j)}}$
are the unknowns associated with $Rx_j$. Denote the solution vector as
\[ {\bf v} =
(u_{11},u_{12},\ldots,u_{1N_0},~u_{21},u_{22},\ldots,u_{2N_0})^{\mathsf T} \]
and rewrite the discretized system as a matrix form
\[ A{\bf v} = \lambda B{\bf v}. \]
We have
\begin{displaymath}
    A = \left[\begin{array}{cc}
        A_{[11]} & A_{[12]} \\
        A_{[21]} & A_{[22]}
    \end{array} \right], ~
    B = \left[ \begin{array}{cc}
        B_{[11]} & B_{[12]} \\
        B_{[21]} & B_{[22]}
    \end{array} \right],
\end{displaymath}
where $A_{[ml]} = (A_{[ml]ij})_{N_0\times N_0}~(m,l=1,2)$ with
\begin{equation}
    \label{eq:nabelAB}
    \begin{array}{rcll}
        A_{[11]ij} &=& {\displaystyle\sum_{R\in G}}
        {\Gamma^{(\nu)}(R)}^*_{11}~a_{\scriptscriptstyle{i,R(j)}},
        &~A_{[12]ij} = {\displaystyle\sum_{R\in G}}
        {\Gamma^{(\nu)}(R)}^*_{12}~a_{\scriptscriptstyle{i,R(j)}},
        \vspace{0.1cm}\\
        A_{[21]ij} &=& {\displaystyle\sum_{R\in G}}
        {\Gamma^{(\nu)}(R)}^*_{21}~a_{\scriptscriptstyle{i,R(j)}},
        &~A_{[22]ij} = {\displaystyle\sum_{R\in G}}
        {\Gamma^{(\nu)}(R)}^*_{22}~a_{\scriptscriptstyle{i,R(j)}}.
    \end{array}
\end{equation}
Entries of $B$ are in the same form as those of $A$ and can be
obtained by substituting $a_{\scriptscriptstyle{i,R(j)}}$ with
$b_{\scriptscriptstyle{i,R(j)}}$.

If symmetry group $G$ is Abelian, each irreducible representation is
one-dimensional and all discretized subproblems are independent.
Otherwise, there exist $\Gamma^{(\nu)}$ with $d_{\nu}>1$ and
the corresponding $d_{\nu}$ discretized subproblems are coupled through
symmetry characteristics.
Thus, no matter $G$ is Abelian or not, we shall solve $n_c$ decoupled
eigenvalue problems,
where $n_c$ is the number of irreducible representations.
And the size of discretized system for the $\nu$-th problem is $d_{\nu}N_0$.

\subsection{Symmetry-adapted bases}
In Section~\ref{sec:fe_discretized}, we shall illustrate the relation between
our approach and the approach that constructs symmetry-adapted bases.
For this purpose, in the current subsection,
we tell how to construct the symmetry-adapted bases,
which is the most critical step in the latter approach.

Consider the weak form of (\ref{eq:ori-problem}):
find $(\lambda, u)\in\mathbb{R}\times V$ such that
\[ a(u, v) = \lambda(u, v) \quad\forall v\in V, \]
where $a(\cdot,\cdot)$ is the associated bilinear form over $V\times V$.

Note that the discussion in this part is not restricted to grid-based
discretizations, but we still use notation $N$ and $N_0$ for brevity.
Suppose that we start from $N$ basis functions $\{\psi\}$ of some type,
which satisfy that for any $R\in G$, $P_R\psi$ is one of the basis functions
when $\psi$ is, i.e.,
the $N$ basis functions are chosen with respect to symmetry group $G$.
For simplicity, like the assumption for grid-based discretizations,
we assume that the $g$ basis functions $\{P_{R}\psi: R\in G\}$ are linearly
independent for any basis function $\psi$.
We see that the number of basis functions in the set
which could produce all $N$ ones by applying $\{R\in G\}$
is $g$ times smaller than $N$. We denote this set by
\[ \{\psi_j:j=1,2,\ldots,N_0\}, \]
then all basis functions are given as
\[ \{ P_R\psi_j:j=1,2,\ldots,N_0,~R\in G\}. \]

For the given $\nu$, we fix some $k\in\{1,2,\ldots,d_{\nu}\}$ and
generate symmetry-adapted bases for the $k$-th subproblem in
(\ref{eq:subproblem}).
This is achieved by applying projection operator $\mathscr{P}^{(\nu)}_{kk}$
on all the basis functions $\{ P_R\psi_j:j=1,2,\ldots,N_0,~R\in G\}$.
Suppose that we obtain $N'$ linearly independent symmetry-adapted bases from
this process and we denote them as $\{\Psi_j:j=1,2,\ldots,N'\}$.
Then for any $l\in\{1,2,\ldots,d_{\nu}\}$, symmetry-adapted bases for
the $l$-th subproblem can be given as
$\{\mathscr{P}^{(\nu)}_{lk}\Psi_j:j=1,2,\ldots,N'\}$.

Consider the $d_{\nu}$ discretized systems under the generated bases.
Matrix elements of the $l$-th discretized system are
\[a(\mathscr{P}^{(\nu)}_{lk}\Psi_j, \mathscr{P}^{(\nu)}_{lk}\Psi_i), ~
(\mathscr{P}^{(\nu)}_{lk}\Psi_j, \mathscr{P}^{(\nu)}_{lk}\Psi_i),\quad
i,j=1,2,\ldots,N'. \]
For each $j\in\{1,2,\ldots,N'\}$, according to Proposition \ref{prop:basis},
$\{\mathscr{P}^{(\nu)}_{lk}\Psi_j:l=1,2,\ldots,d_{\nu}\}$ form a basis for
$\Gamma^{(\nu)}$.
We see from (\ref{eq:two_orth}) and Proposition \ref{prop:Lv} that
all the $d_{\nu}$ discretized systems are the same.
So we only need to solve the discretized system
corresponding to the $k$-th subproblem:
\[ \sum_{j=1}^{N'} a(\Psi_j, \Psi_i)~\alpha_j = \lambda^{(\nu)}
\sum_{j=1}^{N'} (\Psi_j, \Psi_i)~\alpha_j,\quad i = 1,2,\ldots,N', \]
where $\{\alpha_j\}$ are the unknowns.
After calculating $\{\alpha_j\}$, the approximated eigenfunctions
for the $l$-th subproblem can be achieved by
\[ u_l^{(\nu)} = \sum_{j=1}^{N'}\alpha_j\mathscr{P}^{(\nu)}_{lk}\Psi_j,
\quad l=1,2,\ldots,d_{\nu}.\]

Next we show how many symmetry-adapted bases would be constructed for
the $\nu$-$k$ symmetry, i.e., the number $N'$ of linearly independent
symmetry-adapted functions in
\[\mathscr{P}^{(\nu)}_{kk}\{P_R\psi_j:j=1,2,\ldots,N_0,~R\in G\}. \]
And then we give the specific way to obtain these functions.

\begin{theorem}
    \label{thm:basis_real}
    Suppose the original basis functions
    $\{P_R\psi_j:j=1,2,\ldots,N_0,~R\in G\}$ satisfy that
    for each $j\in\{1,2,\ldots,N_0\}$ the $g$ functions in
    $\{P_{R}\psi_j: R\in G\}$ are linearly independent.
    Then for any given $\nu\in\{1,2,\ldots,n_c\}$ and
    $k\in\{1,2,\ldots,d_{\nu}\}$, there are $d_{\nu}N_0$
    symmetry-adapted bases for the $\nu$-$k$ symmetry.
\end{theorem}
\begin{proof}
    We need to prove that there are exactly $d_{\nu}N_0$ linearly independent
    symmetry-adapted functions in
    $\{ \mathscr{P}^{(\nu)}_{kk}P_R\psi_j:j=1,2,\ldots,N_0,~R\in G\}$.

    For any $R\in G$ and $j\in\{1,2,\ldots,N_0\}$, since
    \begin{equation}
        \label{eq:Pkk_PR}
        \mathscr{P}^{(\nu)}_{kk}P_R\psi_j = \frac{d_{\nu}}{g}
        \sum_{R'\in G} {\Gamma^{(\nu)}(R')}^*_{kk} P_{R'}P_R\psi_j
        = \frac{d_{\nu}}{g} \sum_{S\in G}
        {\Gamma^{(\nu)}(SR^{-1})}^*_{kk} P_{S}\psi_j,
    \end{equation}
    we see that $\mathscr{P}^{(\nu)}_{kk}P_R\psi_j$ is a linear combination of
    functions $\{P_{S}\psi_j: S\in G\}$ and the coefficient of $P_{S}\psi_j$
    is $\frac{d_{\nu}}{g}{\Gamma^{(\nu)}(SR^{-1})}^*_{kk}$.
    Obviously, functions in
    $\{\mathscr{P}^{(\nu)}_{kk}P_R\psi_j:j=1,2,\ldots,N_0,~R\in G\}$
    with different $j$ values are linearly independent.
    So we only need to determine the number of symmetry-adapted bases in
    $\{ \mathscr{P}^{(\nu)}_{kk}P_R\psi_j:R\in G\}$ for any given
    $j\in\{1,2,\ldots,N_0\}$.

    Since $\{P_{S}\psi_j: S\in G\}$ are linearly independent
    and $R^{-1}$ runs over all elements of group $G$ when $R$ does,
    (\ref{eq:Pkk_PR}) tells that the number of linearly independent
    functions in $\{ \mathscr{P}^{(\nu)}_{kk}P_R\psi_j: R\in G\}$
    equals to the rank of matrix $C = (C_{mn})_{g\times g}$,
    where $C_{mn} = {\Gamma^{(\nu)}(R_mR_n)}^*_{kk}.$

    We observe that $C$ can be written as
    \begin{displaymath}
        C =
        \left[\begin{array}{ccc}
            {\Gamma^{(\nu)}(R_1)}^*_{k1} & \ldots
            & {\Gamma^{(\nu)}(R_1)}^*_{kd_{\nu}} \\
            {\Gamma^{(\nu)}(R_2)}^*_{k1} & \ldots
            & {\Gamma^{(\nu)}(R_2)}^*_{kd_{\nu}} \\
            \vdots & & \vdots \\
            {\Gamma^{(\nu)}(R_g)}^*_{k1} & \ldots
            & {\Gamma^{(\nu)}(R_g)}^*_{kd_{\nu}}
        \end{array} \right]
        \left[ \begin{array}{ccc}
            {\Gamma^{(\nu)}(R_1)}^*_{1k} & \ldots
            & {\Gamma^{(\nu)}(R_g)}^*_{1k} \\
            {\Gamma^{(\nu)}(R_1)}^*_{2k} & \ldots
            & {\Gamma^{(\nu)}(R_g)}^*_{2k} \\
            \vdots & & \vdots \\
            {\Gamma^{(\nu)}(R_1)}^*_{d_{\nu}k} & \ldots
            & {\Gamma^{(\nu)}(R_g)}^*_{d_{\nu}k}
        \end{array} \right]
        \equiv C_1C_2,
    \end{displaymath}
    where $C_1$ and $C_2$ are $g\times d_{\nu}$ and $d_{\nu}\times g$
    matrices, respectively.
    We obtain from the great orthogonality theorem (\ref{eq:great_orth}) that
    columns of $C_1$ are orthogonal, and so are rows of $C_2$, i.e.,
    \[ \mbox{rank}(C_1) = \mbox{rank}(C_2) = d_{\nu}. \]
    Thus
    \[ \mbox{rank}(C) = \mbox{rank}(C_1C_2) = d_{\nu}, \]
    and we completed the proof.
\qquad\end{proof}

\begin{remark}
    \label{rmk:basis_real}
    For the given $\nu$ and $k$, Theorem \ref{thm:basis_real} indicates that
    there are $d_{\nu}$ symmetry-adapted bases for each
    $j\in\{1,2,\ldots,N_0\}$.
    It remains a problem how to obtain these $d_{\nu}$ functions.
    We see from (\ref{eq:Pkk_PR}) that, whenever the chosen
    $d_{\nu}$ operations $\{R_n\in G: n=1,2,\ldots,d_{\nu}\}$ satisfy that
    the $k$-th columns of matrices
    $\{\Gamma^{(\nu)}(R_n^{-1}): n=1,2,\ldots,d_{\nu}\}$
    are linearly independent,
    $\{ \mathscr{P}^{(\nu)}_{kk}P_{R_n}\psi_j: n=1,2,\ldots,d_{\nu}\}$
    exactly give the $d_{\nu}$ symmetry-adapted bases.
\end{remark}

\subsection{Relation}\label{sec:fe_discretized}
In this part, taking the finite element discretization as an example,
we investigate the relation between our discretized systems and those formed by
the symmetry-adapted bases.

Consider the finite element discretization and denote the basis function
corresponding to any $j\in\{1,2,\ldots,N_0\}$ as $\varphi_j$.
We see from $P_R\varphi_j(x) = \varphi_j(R^{-1}x)$ that
$P_R\varphi_j$ is the basis function corresponding to $R(j)$, i.e.,
\[ P_R\varphi_j = \varphi_{\scriptscriptstyle{R(j)}}. \]

Our discretized systems associated with the finite element basis functions
$\{P_R\varphi_j:j=1,2,\ldots,N_0,~R\in G\}$ are determined by
setting $a_{\scriptscriptstyle{i,R(j)}}$ and $b_{\scriptscriptstyle{i,R(j)}}$
in (\ref{eq:abelAB}) and (\ref{eq:nabelAB}) as
\begin{equation}
    \label{eq:ab_FE}
    a_{\scriptscriptstyle{i,R(j)}} = a(P_R\varphi_j,\varphi_i),\quad
    b_{\scriptscriptstyle{i,R(j)}} = (P_R\varphi_j,\varphi_i).
\end{equation}
Now we turn to study the discretized systems from the approach that constructs
symmetry-adapted bases, and obtain the relation between the two approaches.

In the case of $d_{\nu}=1$, we apply projection operator $\mathscr{P}^{(\nu)}$
on all the finite element basis functions to construct the symmetry-adapted
bases. We see from Theorem \ref{thm:basis_real} that for each
$j\in\{1,2,\ldots,N_0\}$,
$\{ \mathscr{P}^{(\nu)}P_R\varphi_j: R\in G\}$
give one symmetry-adapted basis function.
According to Remark \ref{rmk:basis_real},
we can choose $R=E$ to get all the $N_0$ symmetry-adapted bases as follows
\[ \Phi_j = \mathscr{P}^{(\nu)}\varphi_j,~j=1,2,\ldots,N_0. \]
The discretized system under these bases then becomes
\[ \sum_{j=1}^{N_0}a(\Phi_j,\Phi_i)c_j = \tilde{\lambda}
\sum_{j=1}^{N_0}(\Phi_j,\Phi_i)c_j,\quad i=1,2,\ldots,N_0, \]
where $\{c_j\}$ are the unknowns. Equivalently,
\[ \widetilde{A}\tilde{\bf u} = \tilde{\lambda}\widetilde{B}\tilde{\bf u}, \]
where $\tilde{\bf u} = (c_1,c_2,\ldots,c_{N_0})^{\mathsf T}$ and
\begin{equation}
    \label{eq:sa_abelAB}
    \begin{array}{rcll}
        \widetilde{A} &=& (\widetilde{A}_{ij})_{N_0\times N_0},
        &~\widetilde{A}_{ij} = \frac{1}{g} {\displaystyle\sum_{R\in G}}
        {\Gamma^{(\nu)}(R)}^*~a(P_R\varphi_j,\varphi_i),
        \vspace{0.1cm}\\
        \widetilde{B} &=& (\widetilde{B}_{ij})_{N_0\times N_0},
        &~\widetilde{B}_{ij} =\frac{1}{g} {\displaystyle\sum_{R\in G}}
        {\Gamma^{(\nu)}(R)}^* (P_R\varphi_j,\varphi_i).
    \end{array}
\end{equation}
Comparing (\ref{eq:sa_abelAB}) with (\ref{eq:abelAB}) and using
(\ref{eq:ab_FE}), we obtain
\[ \widetilde{A} = \frac{1}{g}A, ~\widetilde{B} = \frac{1}{g}B. \]
Thus, in the case of $d_{\nu}=1$, there holds
\[ \lambda = \tilde{\lambda},~{\bf u} = \tilde{\bf u}. \]

In the case of $d_{\nu}=2$, there are two subproblems in (\ref{eq:subproblem}).
We choose $k=1$ and apply projection operator $\mathscr{P}^{(\nu)}_{11}$ on
all the finite element basis functions to construct symmetry-adapted bases
for the first subproblem.
Theorem \ref{thm:basis_real} tells that for each
$j\in\{1,2,\ldots,N_0\}$,
$\{ \mathscr{P}^{(\nu)}_{11} P_R\varphi_j:R\in G \}$
give $d_{\nu}=2$ symmetry-adapted bases.
According to Remark \ref{rmk:basis_real},
we choose identity operation $E$ and another $S\in G$ which satisfy that
the first columns of matrices
$\{ \Gamma^{(\nu)}(E),~\Gamma^{(\nu)}(S^{-1}) \}$
are linearly independent. Then
\[ \{ \mathscr{P}^{(\nu)}_{11}\varphi_j,
\mathscr{P}^{(\nu)}_{11}P_S\varphi_j:j=1,2,\ldots,N_0 \} \]
give all the $2N_0$ bases adapted to the $\nu$-$1$ symmetry
as follows
\[ (\Phi_1, \ldots, \Phi_{N_0}, \Psi_1,\ldots,\Psi_{N_0}) =
(\mathscr{P}^{(\nu)}_{11}\varphi_1,\ldots,\mathscr{P}^{(\nu)}_{11}\varphi_{N_0},
\mathscr{P}^{(\nu)}_{11}P_S\varphi_1,\ldots,
\mathscr{P}^{(\nu)}_{11}P_S\varphi_{N_0}). \]
The discretized system under these bases is
\[
\left\{
\begin{array}{c}
    {\displaystyle\sum_{j=1}^{N_0}}~a(c_{1j}\Phi_j+c_{2j}\Psi_j, \Phi_i) =
    \tilde{\lambda}~
    {\displaystyle\sum_{j=1}^{N_0}}~(c_{1j}\Phi_j+c_{2j}\Psi_j, \Phi_i),
    \quad i=1,2,\ldots,N_0,
    \vspace{0.1cm}\\
    {\displaystyle\sum_{j=1}^{N_0}}~a(c_{1j}\Phi_j+c_{2j}\Psi_j, \Psi_i) =
    \tilde{\lambda}~
    {\displaystyle\sum_{j=1}^{N_0}}~(c_{1j}\Phi_j+c_{2j}\Psi_j, \Psi_i),
    \quad i=1,2,\ldots,N_0,
\end{array}\right.
\]
where $\{c_{1j}, c_{2j}\}$ represent the unknowns. Equivalently,
\[ \widetilde{A}\tilde{\bf v} = \tilde{\lambda}\widetilde{B}\tilde{\bf v}, \]
where $\tilde{\bf v} =
(c_{11},c_{12},\ldots,c_{1N_0},~c_{21},c_{22},\ldots,c_{2N_0})^{\mathsf T}$
and
\begin{displaymath}
    \widetilde{A} = \left[ \begin{array}{cc}
        \widetilde{A}_{[11]} & \widetilde{A}_{[12]} \vspace{0.1cm}\\
        \widetilde{A}_{[21]} & \widetilde{A}_{[22]}
    \end{array} \right], ~
    \widetilde{B} = \left[ \begin{array}{cc}
        \widetilde{B}_{[11]} & \widetilde{B}_{[12]} \vspace{0.1cm}\\
        \widetilde{B}_{[21]} & \widetilde{B}_{[22]}
    \end{array} \right].
\end{displaymath}
A simple calculation shows
\begin{eqnarray*}
    \widetilde{A}_{[11]} &=& \frac{2}{g} A_{[11]}, \\
    \widetilde{A}_{[12]} &=& \frac{2}{g} \left(
    \Gamma^{(\nu)}(S)_{11} A_{[11]} +
    \Gamma^{(\nu)}(S)_{12} A_{[12]} \right), \\
    \widetilde{A}_{[21]} &=& \frac{2}{g} \left(
    {\Gamma^{(\nu)}(S)}^*_{11} A_{[11]} +
    {\Gamma^{(\nu)}(S)}^*_{12} A_{[21]} \right), \\
    \widetilde{A}_{[22]} &=& \frac{2}{g} \Big\{
    {\Gamma^{(\nu)}(S)}_{11} \left(
    {\Gamma^{(\nu)}(S)}^*_{11}A_{[11]} +{\Gamma^{(\nu)}(S)}^*_{12}A_{[21]}
    \right) \\
    &&+~{\Gamma^{(\nu)}(S)}_{12} \left(
    {\Gamma^{(\nu)}(S)}^*_{11}A_{[12]} +{\Gamma^{(\nu)}(S)}^*_{12}A_{[22]}
    \right) \Big\}.
\end{eqnarray*}
Let
\begin{displaymath}
    Q_l =
    \left[\begin{array}{cc}
        I_{N_0\times N_0} & {\bf 0}_{N_0\times N_0} \vspace{0.2cm}\\
        {\Gamma^{(\nu)}(S)}^*_{11}I_{N_0\times N_0} &
        {\Gamma^{(\nu)}(S)}^*_{12}I_{N_0\times N_0}
    \end{array} \right], ~
    Q_r =
    \left[\begin{array}{cc}
        I_{N_0\times N_0} &
        \Gamma^{(\nu)}(S)_{11}I_{N_0\times N_0} \vspace{0.2cm}\\
        {\bf 0}_{N_0\times N_0} &
        \Gamma^{(\nu)}(S)_{12}I_{N_0\times N_0}
    \end{array} \right],
\end{displaymath}
we have
\[ Q_l A Q_r = \frac{g}{2} \widetilde{A}. \]
Similarly
\[ Q_l B Q_r = \frac{g}{2} \widetilde{B}. \]
Thus, in the case of $d_{\nu}=2$, we get
\[ \lambda = \tilde{\lambda},\quad{\bf v} = Q_r\tilde{\bf v}, \]
i.e.,
\[ u_{1j} = c_{1j} + {\Gamma^{(\nu)}(S)}_{11}c_{2j}, ~
u_{2j} = {\Gamma^{(\nu)}(S)}_{12}c_{2j},\quad j=1,2,\ldots,N_0. \]

Consider a given $\nu\in\{1,2,\ldots,n_c\}$, the approach that
constructs symmetry-adapted bases seems to have an obvious advantage that
the $d_{\nu}$ subproblems are decoupled.
Theorem \ref{thm:basis_real} tells that the number of symmetry-adapted
bases for each subproblem is in fact $d_{\nu}N_0$. Therefore,
the coupled eigenvalue problem appeared in our decomposition approach
is not an induced complexity, but some reflection of the intrinsic property
of symmetry-based decomposition.

Solving subproblems instead of the original eigenvalue problem shall
reduce the computational overhead and memory requirement to a large extent.
The eigenvalues to be computed are distributed among subproblems, i.e.,
a smaller number of eigenpairs are required for each subproblem.
And the decomposed problems can be solved in a small subdomain.
Moreover, as indicated in Section~\ref{sec:formulation},
there is a possibility to improve the spectral separation,
which would accelerate convergence of iterative diagonalization.
In the next section, we shall propose a way to analyze the practical
decrease in the computational cost.

\section{Complexity and performance analysis}\label{sec:cost}
The advantage of solving subproblems (\ref{eq:subproblem1})
instead of the original problem (\ref{eq:ori-problem}) is
the reduction in computational overhead.
Based on a complexity analysis, we quantize this reduction and
present a way to analyze the practical speedup in CPU time.

\subsection{Complexity analysis}
Computational complexity is the dominant part of computational overhead
when the size of problem becomes sufficiently large.
So the fundamental step of complexity analysis is to figure out
the computational cost in floating point operations (flops).

In our computation, the algebraic eigenvalue problem will be solved by
the implicitly restarted Lanczos method (IRLM) implemented in ARPACK package
\cite{lehoucq98}.
Our complexity analysis will be based on IRLM,
whereas it can be extended to other iterative diagonalization methods.

Total flops of an iterative method are the product of
the number of iteration steps and the number of flops per iteration.
We shall analyze the number of flops per iteration, for which purpose
we represent the procedure of IRLM as Algorithm \ref{alg:IRLM} as follows.

\begin{algorithm}
    \KwIn{Maximum number of iteration steps;
    The $m$-step Lanczos Factorization $AV_m=V_mH_m+f_me_m^T$.}

    \BlankLine
    \Repeat{Convergence {\bf or} the number of iteration steps exceeded
    the maximum one} {
        Compute the Schur decomposition of symmetric tridiagonal matrix
        $H_m$ and select the set of $l$ shifts $\mu_1,\mu_2,\dots,\mu_l$\;

        $q^T\leftarrow e^T_m$\;

        \For{$j=1,2,\dots,l$}{

        $H_m-\mu_jI = Q_jR_j$, $H_m\leftarrow R_jQ_j + \mu_jI$\;

        $V_m\leftarrow V_mQ_j$, $q^H\leftarrow q^HQ_j$\;

        }

        $f_k\leftarrow v_{k+1}\hat{\beta}_k+f_m\sigma_k,
        V_k\leftarrow V_m(1:n,1:k),
        H_k\leftarrow H_m(1:k,1:k)$\;

        Beginning with the $k$-step Lanczos factorization
        $AV_k=V_kH_k+f_ke_k^T$,
        apply $l$ additional steps of the Lanczos process to obtain
        a new $m$-step Lanczos factorization $AV_m=V_mH_m+f_me_m^T$\;
     }
    \caption{An implicitly restarted Lanczos method \label{alg:IRLM}}
\end{algorithm}

\begin{table}[htpb]
    \caption{Notation in Algorithm \ref{alg:IRLM}.}
    \begin{center} \footnotesize
    \begin{tabular}{|c|c|}\hline
        Notation & Description \\\hline
        \multirow{2}{*}{$m$} & the maximum dimension of the Krylov subspace,\\
        & twice the number of required eigenpairs plus 5
        in our computation\\\hline
        $l$ & the number of Lanczos factorization steps, s.t. $m=k+l$ \\\hline
        \multirow{2}{*}{$A$} & the (sparse) matrix size of $n\times n$, \\
        & arising from the grid-based discretization
        of (\ref{eq:ori-problem}) or (\ref{eq:subproblem1}) \\\hline
        \multirow{2}{*}{$V_m$} & the matrix size of $n\times m$, \\
        & made of $m$ column vectors as the basis of the Krylov subspace
        \\\hline
        $H_m$ & the symmetric tridiagonal matrix size of $m\times m$ \\\hline
        \multirow{2}{*}{$f_m$} & the column vector size of $n$, \\
        & the residual vector after $m$ steps of Lanczos factorization \\\hline
        $e_m$ & the unit column vector size of $m$,
        in which the $m$-th component is one\\\hline
        $R_j$ & the upper triangular matrix size of $m\times m$ \\\hline
        $Q_j$ & the unitary matrix size of $m\times m$ \\\hline
        $v_{k+1}$ & the $(k+1)$-th column vector of $V_m$ \\\hline
        $\hat{\beta}_k$ & $H_m(k+1,k)$ \\\hline
        $\sigma_k$ & the $k$-th component of vector $q$ \\\hline
    \end{tabular}
    \end{center}
    \label{tab:IRLM}
\end{table}

Table \ref{tab:IRLM} is a supplementary remark to Algorithm \ref{alg:IRLM}.
In Algorithm \ref{alg:IRLM}, Step 2 is the Schur decomposition of $H_m$,
and consumes about $6m^2$ flops \cite{golub96}.
Steps 4 to 7 do $l$-step QR iteration with shifts.
Note that each $Q_j$ is the product of $(m-1)$ Givens transformations,
we have that Step 5 costs $8m(m-1)$ flops since applying one Givens
transformation to a matrix only changes two rows or columns of the matrix.
And for the same reason, Step 6 costs $4(m-1)(n+1)$ flops.
Consequently Steps 4 to 7 consume $4l(m-1)(2m+n+1)$ flops.
Regardless of BLAS-1 operations, we do $l$ matrix-vector multiplication
operations at Step 9.

Besides order $n$ of the matrix, the flops of one matrix-vector multiplication
also depend on the order of finite difference or finite elements.
If the shift-invert mode in ARPACK is employed to solve the generalized
eigenvalue problem arising from the finite element discretization,
the matrix-vector multiplication will be realized by some iterative linear
solver.
So we cannot figure out accurately the flops per matrix-vector multiplication
but represent it as $\mathcal{O}(n)$.

In total, the computational overhead per IRLM iteration can be estimated as
\[ 6m^2+l\left(4(m-1)(2m+n+1)+\mathcal{O}(n)\right) \]
flops. In general, order $n$ of the matrix is much more than $m$
for grid-based discretizations.
So the majority of flops per IRLM iteration is
\begin{equation}
    f(l,m,n) = l\left(4mn+\mathcal{O}(n)\right).
    \label{eq:flops_iteration}
\end{equation}

In order to make clear the reduction in flops per iteration
from solving subproblems instead of the original eigenvalue problem,
we divide the flops per iteration into two parts.
One is required by $l$-step $QR$ iteration, and the other is spent on
$l$ operations of matrix-vector multiplication.
We denote them by $f_1$ and $f_2$ respectively and rewrite
(\ref{eq:flops_iteration}) as follows
\begin{equation}
    \label{eq:flops_iteration_parts}
    f(l,m,n) = f_1(l,m,n) + f_2(l,m,n),
\end{equation}
where $f_1(l,m,n)=4lmn$ and $f_2(l,m,n)=\mathcal{O}(ln)$.

In solving the original eigenvalue problem (\ref{eq:ori-problem}),
the major flops per IRLM iteration can be accounted as
(\ref{eq:flops_iteration}) or (\ref{eq:flops_iteration_parts}) with $n=N$.
In the decomposition approach, as discussed in Section~\ref{sec:discretized_sys},
we shall solve $n_c$ decoupled eigenvalue problems,
and the size of discretized system for the $\nu$-th problem is $d_{\nu}N_0$.
In solving the $\nu$-th problem (\ref{eq:subproblem1}),
$m$ is reduced to $m/\theta_1$, $N$ to $d_{\nu}N/g$, and $l$ to $l/\theta_2$,
where $g$ is the order of finite group $G$, $\theta_1>1$ and
$\theta_2\approx \theta_1$ because $l$ is almost proportional to $m$
in Algorithm \ref{alg:IRLM}.
We shall explain in Section~\ref{sec:distrib_eigen} that the number of
required eigenpairs for each subproblem is set as the same in the computation,
so all the subproblems have an identical $\theta_1$.
Thus, the majority of total flops per iteration for all $n_c$
decomposed eigenvalue problems is
\begin{eqnarray}
    \sum_{\nu=1}^{n_c}
    f(\frac{l}{\theta_2},\frac{m}{\theta_1},\frac{d_{\nu}N}{g}) &=&
    \sum_{\nu=1}^{n_c}
    \left(f_1(\frac{l}{\theta_2},\frac{m}{\theta_1},\frac{d_{\nu}N}{g}) +
    f_2(\frac{l}{\theta_2},\frac{m}{\theta_1},\frac{d_{\nu}N}{g})\right)
    \nonumber\\
    &=& \frac{n_{sub}}{g}\left(\frac{1}{\theta_1\theta_2}f_1(l,m,N) +
    \frac{1}{\theta_2}f_2(l,m,N)\right),
    \label{eq:reduced_flops}
\end{eqnarray}
where $n_{sub} = \sum_{\nu=1}^{n_c}d_{\nu}$ is the number of subproblems.

As mentioned in Section~\ref{sec:discretization}, the decomposition approach
saves the computational cost of solving the eigenvalue problem.
Now the reduction can be characterized by (\ref{eq:reduced_flops}).

\subsection{Performance analysis}
In (\ref{eq:reduced_flops}), the order of factors for $f_1$ and $f_2$ differs,
so the practical speedup in CPU time cannot be properly estimated from
(\ref{eq:reduced_flops}).
We introduce the CPU time ratio $\omega$ of the matrix-vector multiplications
to the whole IRLM process in solving the original eigenvalue problem
(\ref{eq:ori-problem}). It is an a posteriori parameter which
screens affects of implementation, the runtime environment, as well as
the specific linear solver for the shift-invert mode.
Besides, testing for $\omega$ is feasible as the operation of
matrix-vector multiplication is usually provided by users.

Applying the symmetry-based decomposition approach instead of
solving (\ref{eq:ori-problem}) directly, we can show the speedup in
CPU time of one IRLM iteration as follows:
\[ s(\theta_1,\theta_2,\omega) = \frac{\frac{1}{\omega} }
{\frac{n_{sub}}{g} \left(\frac{1}{\theta_1\theta_2}\frac{1-\omega}{\omega}
+\frac{1}{\theta_2}\right)} = \frac{g\theta_1\theta_2}
{n_{sub}\left(1+(\theta_1-1)\omega\right)}. \]
That is
\begin{equation}
    \label{eq:speedup_rough}
    s(\theta_1,\theta_2,\omega) \approx
    \frac{g\theta_1^2}{n_{sub}\left(1+(\theta_1-1)\omega\right)}.
\end{equation}

In practice, $\theta_2$ is actually determined by the internal configurations
of algebraic eigenvalue solvers. So we prefer to use (\ref{eq:speedup_rough})
to predict the CPU time speedup before solving subproblems
(\ref{eq:subproblem1}).

In Section~\ref{sec:numer}, the validation of (\ref{eq:speedup_rough})
will be well supported by our numerical experiments.
Moreover, this performance analysis implies that the speedup
will be amplified when more eigenpairs are required and
a consequent decrease in $\omega$ is very likely.
Therefore, the symmetry-based decomposition will be attractive for
large-scale eigenvalue problems.

\section{Practical issues}\label{sec:impl}
In this section, we address some key issues in the implementation of
the symmetry-based decomposition approach under grid-based discretizations.

\subsection{Implementation of symmetry characteristics}
Symmetry characteristics play a critical role in the decomposition approach,
so it is important to preserve and realize symmetry characteristics
for discretized eigenfunctions.

For all the degrees of freedom not lying on symmetry elements,
the implementation of symmetry characteristics is straightforward
with grid-based discretizations. If $x\in\Omega$ is a degree of freedom
lying on the symmetry element corresponding to operation $R\in G$,
the symmetry characteristic
\[ u_l^{(\nu)}(Rx) = \sum_{m=1}^{d_{\nu}}
{\Gamma^{(\nu)}(R)}^*_{lm} u_{m}^{(\nu)}(x) \]
reduces to
\[ u_l^{(\nu)}(x) = \sum_{m=1}^{d_{\nu}}
{\Gamma^{(\nu)}(R)}^*_{lm} u_{m}^{(\nu)}(x). \]

If $\mathrm{det}\left(\Gamma^{(\nu)}(R) - I_{d_{\nu}\times d_{\nu}}\right)
\neq 0$, then all values $u_1^{(\nu)}(x),\dots,u_{d_{\nu}}^{(\nu)}(x)$
are zeros.
Otherwise, we have to find the independent ones out of
$u_1^{(\nu)}(x),\dots,u_{d_{\nu}}^{(\nu)}(x)$ and treat them as
additional degrees of freedom.

In our computation, we discretize the problem on a tensor-product grid
associated with the symmetry group.
Currently, for simplicity, we use symmetry groups with symmetry elements
on the coordinate planes, and prevent degrees of freedom from lying on
the symmetry elements,
by imposing an odd number of partition in each direction and
using finite elements of odd orders.

\subsection{Distribution of required eigenpairs among subproblems}
\label{sec:distrib_eigen}
The required eigenpairs of the original eigenvalue problem
(\ref{eq:ori-problem}) are distributed among associated subproblems,
and the number of eigenpairs required by each subproblem can be almost
reduced by as many times as the number of subproblems.
However, we are not able to see in advance the symmetry properties of
eigenfunctions corresponding to required eigenvalues.
Thus we have to consider some redundant eigenvalues for each subproblem.

We suppose to solve the first $N_e$ smallest eigenvalues of
the original problem.
First we set the number of eigenvalues to be computed for each subproblem as
$\frac{N_e}{n_{sub}}$ plus redundant $n_a$ eigenvalues,
where $n_{sub}=\sum_{\nu=1}^{n_c}d_{\nu}$ is the number of subproblems.
After solving the subproblems, we gather eigenvalues from all subproblems
and sort them in the ascending order.
After taking $N_e$ smallest eigenvalues, we check which subproblems
the remaining eigenvalues belong to.
If there is no eigenvalue left for some subproblem,
the number of computed eigenvalues for this subproblem is probably not enough.
Subsequently we restart computing the subproblem with an increased number of
required eigenpairs.

\subsection{Two-level parallel implementation}\label{sec:parallel_impl}
We have addressed in Section~\ref{sec:discretization} that
the $n_c$ decomposed problems are independent to each other and
can be solved simultaneously.
Accordingly we have a two-level parallel implementation illustrated by
Figure~\ref{fig:two_level}.
At the first level, we dispatch the $n_c$ decomposed problems among
groups of processors. At the second level, we distribute the grids among
each group of processors.
Since eigenfunctions of different subproblems are naturally orthogonal,
there is no communication between different groups of processors
during solving the eigenvalue problem.
Such two-level or multi-level parallelism is likely appreciable for
the architecture hierarchy of modern supercomputers.
We shall see in Section~\ref{sec:comm_cost} that the two-level parallel
implementation does reduce the communication cost.

\begin{figure}[htpb]
    \begin{center}
        \includegraphics[width=0.9\textwidth]{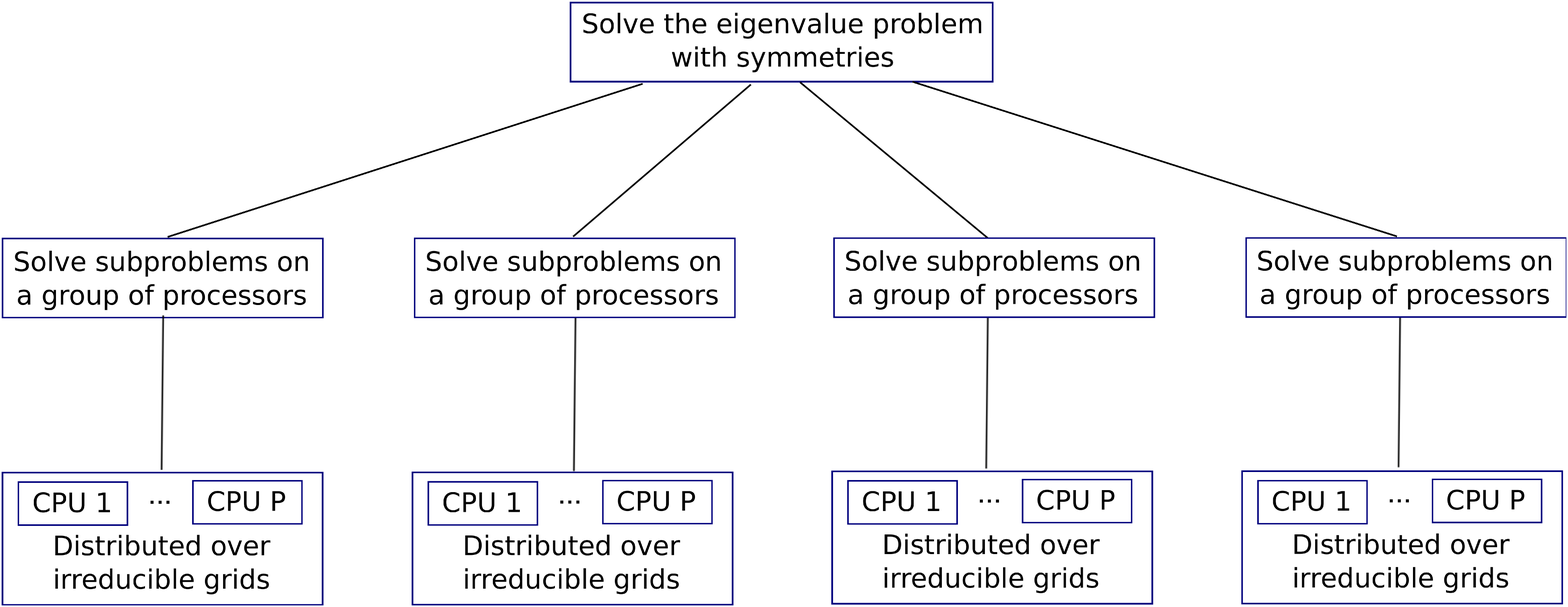}
    \end{center}
    \caption{Schematic illustration of two-level parallel implementation
    for solving the eigenvalue problem with symmetries.
    Actually, the number of processors in each group can be in
    proportion to $d_{\nu}N_0$, which is size of the $\nu$-th discretized
    system.}
    \label{fig:two_level}
\end{figure}

\section{Numerical tests and applications}\label{sec:numer}
In this section, we present some numerical examples arising from
quantum mechanics to validate the implementation and
illustrate the efficiency of the decomposition approach.
We use hexahedral finite element discretizations and
consider the crystallographic point groups of which symmetry operations
keep the hexahedral grids invariant.
We solve the matrix eigenvalue problem using subroutines of ARPACK.
Our computing platform is the LSSC-III cluster provided by
State Key Laboratory of Scientific and Engineering Computing (LSEC),
Chinese Academy of Sciences.

\subsection{Validation of implementation}
First we validate the implementation of the decomposition approach.
Consider the harmonic oscillator equation which is
a basic quantum eigenvalue problem as follows
\begin{equation}
    \label{eq:osc}
    -\frac{1}{2}\Delta u + \frac{1}{2} |x|^2 u = \lambda u
    \quad\mbox{in}~\mathbb{R}^3.
\end{equation}
The exact eigenvalues are given as
\[\lambda_{k,m,n} = k+m+n+1.5,~~k,m,n = 0,1,2,\ldots.\]
The computation can be done in a finite domain with zero boundary condition
since the eigenfunctions decay exponentially.
We set $\Omega = (-5.0,5.0)^3$ in our calculations and
solve the first 10 eigenvalues.

Obviously, the system has all the cubic symmetries.
As representatives, we test Abelian subgroup $D_{2h}$ and
non-Abelian subgroups $D_4$ and $D_{2d}$.
Table \ref{table:IR} gives the irreducible representation matrices of
these groups \cite{cornwell97}, where
\begin{eqnarray*}
S_1=\left[\begin{array}{cc} 1 & 0 \\ 0 & 1 \end{array}\right],\quad
S_2=\left[\begin{array}{cc}-1 & 0 \\ 0 &-1 \end{array}\right],\quad
S_3=\left[\begin{array}{cc} 0 &-1 \\ 1 & 0 \end{array}\right],\quad
S_4=\left[\begin{array}{cc} 0 & 1 \\-1 & 0 \end{array}\right],\\
S_5=\left[\begin{array}{cc} 1 & 0 \\ 0 &-1 \end{array}\right],\quad
S_6=\left[\begin{array}{cc}-1 & 0 \\ 0 & 1 \end{array}\right],\quad
S_7=\left[\begin{array}{cc} 0 & 1 \\ 1 & 0 \end{array}\right],\quad
S_8=\left[\begin{array}{cc} 0 &-1 \\-1 & 0 \end{array}\right].
\end{eqnarray*}
The hexahedral grids can be kept invariant under the three groups.

\begin{table}[htbp]
    \caption{Representation matrices of Abelian group $D_{2h}$ and
    non-Abelian groups $D_4$ and $D_{2d}$.
    All the three groups have 8 symmetry operations, i.e., order $g=8$.
    Abelian group $D_{2h}$ has $n_c=8$ one-dimensional irreducible
    representations, and both the two non-Abelian groups have $n_c=5$
    irreducible representations, one of which is two-dimensional.
    A description about the notation of symmetry operations in the table
    is given in Appendix A.}
    \begin{center} \footnotesize
    \begin{tabular}{|c|cccccccc|}
        \hline
        $D_{2h}$ & $E$ & $C_{2x}$ & $C_{2e}$ & $C_{2f}$
        & $I$ & $IC_{2x}$ & $IC_{2e}$ & $IC_{2f}$ \\
        \hline
        $\Gamma^{(1)}$ & 1 & 1 & 1 & 1 & 1 & 1 & 1 & 1 \\
        $\Gamma^{(2)}$ & 1 & 1 &-1 &-1 & 1 & 1 &-1 &-1 \\
        $\Gamma^{(3)}$ & 1 &-1 & 1 &-1 & 1 &-1 & 1 &-1 \\
        $\Gamma^{(4)}$ & 1 &-1 &-1 & 1 & 1 &-1 &-1 & 1 \\
        $\Gamma^{(5)}$ & 1 & 1 & 1 & 1 &-1 &-1 &-1 &-1 \\
        $\Gamma^{(6)}$ & 1 & 1 &-1 &-1 &-1 &-1 & 1 & 1 \\
        $\Gamma^{(7)}$ & 1 &-1 & 1 &-1 &-1 & 1 &-1 & 1 \\
        $\Gamma^{(8)}$ & 1 &-1 &-1 & 1 &-1 & 1 & 1 &-1 \\
        \hline \hline
        $D_4$ & $E$ & $C_{2y}$ & $C_{4y}$ & $C^{-1}_{4y}$
        & $C_{2x}$ & $C_{2z}$ & $C_{2c}$ & $C_{2d}$ \\
        \hline
        $\Gamma^{(1)}$ & 1 & 1 & 1 & 1 & 1 & 1 & 1 & 1 \\
        $\Gamma^{(2)}$ & 1 & 1 &-1 &-1 & 1 & 1 &-1 &-1 \\
        $\Gamma^{(3)}$ & 1 & 1 & 1 & 1 &-1 &-1 &-1 &-1 \\
        $\Gamma^{(4)}$ & 1 & 1 &-1 &-1 &-1 &-1 & 1 & 1 \\
        $\Gamma^{(5)}$ & $S_1$ & $S_2$ & $S_3$ & $S_4$
        & $S_5$ & $S_6$ & $S_7$ & $S_8$ \\
        \hline \hline
        $D_{2d}$ & $E$ & $C_{2y}$ & $IC_{4y}$ & $IC^{-1}_{4y}$
        & $IC_{2x}$ & $IC_{2z}$ & $C_{2c}$ & $C_{2d}$ \\
        \hline
        $\Gamma^{(1)}$ & 1 & 1 & 1 & 1 & 1 & 1 & 1 & 1 \\
        $\Gamma^{(2)}$ & 1 & 1 &-1 &-1 & 1 & 1 &-1 &-1 \\
        $\Gamma^{(3)}$ & 1 & 1 & 1 & 1 &-1 &-1 &-1 &-1 \\
        $\Gamma^{(4)}$ & 1 & 1 &-1 &-1 &-1 &-1 & 1 & 1 \\
        $\Gamma^{(5)}$ & $S_1$ & $S_2$ & $-S_3$ & $-S_4$
        & $-S_5$ & $-S_6$ & $S_7$ & $S_8$ \\
        \hline
    \end{tabular}
    \end{center}
    \label{table:IR}
\end{table}

According to Theorem \ref{thm:subproblem}, we can decompose
the original eigenvalue problem (\ref{eq:osc}) as follows:
\begin{enumerate}
    \item Applying $D_{2h}$, we have 8 completely decoupled subproblems.
    \item Applying $D_4$ or $D_{2d}$, we have 6 subproblems and
        two of them corresponding to representation $\Gamma^{(5)}$ are coupled
        eigenvalue problems.
\end{enumerate}
Figure \ref{fig:omega0} illustrates the irreducible subdomain $\Omega_0$
in which subproblems are solved.
The volume of $\Omega_0$ is one eighth of $\Omega$ for all the three groups.

\begin{figure}[htpb]
    \begin{center}
        \includegraphics[width=0.8\textwidth]{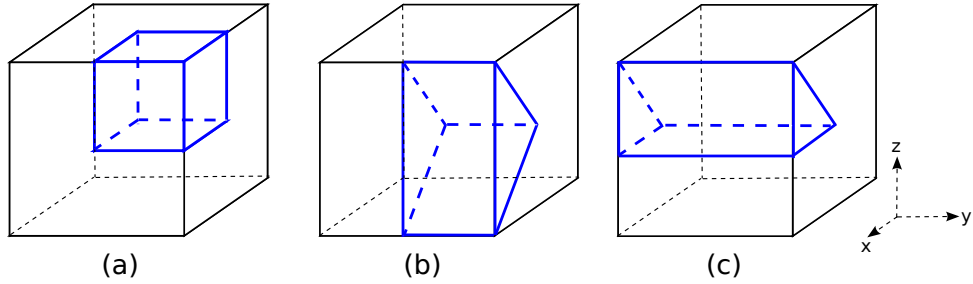}
    \end{center}
    \caption{Illustration of irreducible subdomain $\Omega_0$.
    (a) For $D_{2h}$ it is a small cube;
    (b) For $D_4$, a triangular prism;
    (c) For $D_{2d}$ also a triangular prism, with a different shape.}
    \label{fig:omega0}
\end{figure}

We employ trilinear finite elements to solve these eigenvalue subproblems,
and see from the convergence rate of eigenvalues that the implementation
is correct.
Taking non-Abelian group $D_4$ for instance,
we exhibit errors in eigenvalue approximations obtained from
solving the subproblems in Figure~\ref{fig:conv}.
And the $h^2$-convergence rate can be observed.

Moreover, in Table \ref{table:osc}, we list the $\nu$-$l$ symmetries of
computed eigenfunctions from solving the subproblems.
We observe that the required 10 eigenpairs are distributed over
subproblems, i.e., each subproblem only needs to solve a smaller number
of eigenpairs.

\begin{figure}[htpb]
    \begin{center}
        \includegraphics[width=0.6\textwidth]{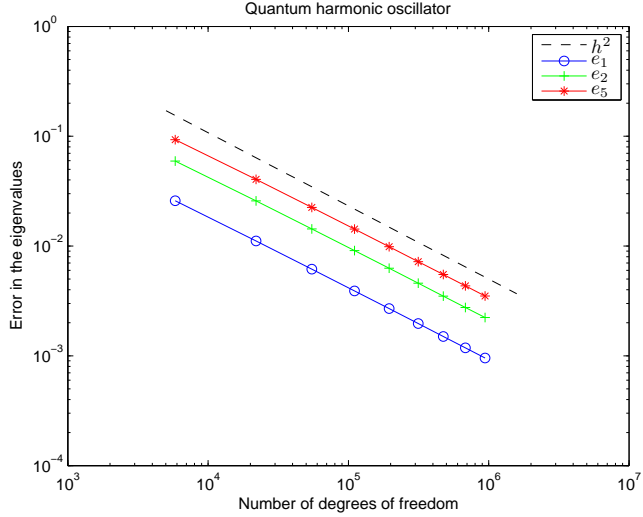}
    \end{center}
    \caption{Errors in the eigenvalue approximations from solving
    6 subproblems associated with non-Abelian group $D_4$ using
    trilinear finite elements.
    Errors in the first three different eigenvalues 1.5, 2.5 and 3.5
    are labeled as $e_1$, $e_2$ and $e_5$, respectively.
    The $h^2$-convergence rate can be observed.}
    \label{fig:conv}
\end{figure}

\begin{table}[htbp]
    \caption{The $\nu$-$l$ symmetries of the first 10 computed eigenfunctions
    from solving subproblems.
    The $\nu$-$l$ values indicate which subproblem each eigenfunction
    belongs to, where $\nu\in\{1,2,\ldots,n_c\}$ and
    $l\in\{1,2,\ldots,d_{\nu}\}$.
    In the case of $D_{2h}$, all the $l$ values are
    1 because it is an Abelian group and all irreducible representations are
    one-dimensional, i.e., $d_{\nu}=1$ for all $\nu=1,2,\ldots,8$.}
    \begin{center} \footnotesize
    \begin{tabular}{|c|c|cccccccccc|}
        \hline
        & & $u_1$&$u_2$&$u_3$&$u_4$&$u_5$&$u_6$&$u_7$&$u_8$&$u_9$&$u_{10}$ \\
        \hline
        \multirow{2}{*}{$D_{2h}$}
        & $\nu$ & 1 & 8 & 6 & 7 & 3 & 4 & 2 & 1 & 1 & 1 \\
        & $l$   & 1 & 1 & 1 & 1 & 1 & 1 & 1 & 1 & 1 & 1 \\ \hline
        \multirow{2}{*}{$D_4$}
        & $\nu$ & 1 & 3 & 5 & 5 & 5 & 5 & 4 & 1 & 2 & 1 \\
        & $l$   & 1 & 1 & 1 & 2 & 1 & 2 & 1 & 1 & 1 & 1 \\ \hline
        \multirow{2}{*}{$D_{2d}$}
        & $\nu$ & 1 & 5 & 5 & 2 & 4 & 5 & 5 & 1 & 2 & 1 \\
        & $l$   & 1 & 1 & 2 & 1 & 1 & 1 & 2 & 1 & 1 & 1 \\ \hline
    \end{tabular}
    \end{center}
    \label{table:osc}
\end{table}

\subsection{Reduction in computational cost}
Taking Abelian group $D_{2h}$ as an example, we compare the computational cost of solving the original eigenvalue problem
(\ref{eq:osc}) with that of solving 8 subproblems.
We compute the first 110 eigenvalues of the original eigenvalue problem.
And it is sufficient to solve the first 22 eigenvalues of each subproblem.
In order to illustrate and analyze the saving in computational cost,
we launch the tests on a single CPU core.

In Table \ref{table:cost-fe1},
we present statistics from trilinear finite element discretizations.
We see that the average CPU time of a single iteration
during solving the original problem (\ref{eq:osc}) is 42.29 seconds
while that of solving 8 subproblems is 3.61 seconds \footnote{
We count the average CPU time of a single iteration for each subproblem and
then accumulate them. Taking Table \ref{table:cost-fe1} for example,
we have that
$3.61=\frac{8.21}{18}+\frac{9.95}{22}+\frac{9.83}{22}+\frac{9.46}{21}
+\frac{8.23}{18}+\frac{9.44}{21}+\frac{9.03}{20}+\frac{9.34}{21}$}.
In Table \ref{table:cost-fe3},
we present statistics from tricubic finite element discretizations.
We observe that the average CPU time of a single iteration
during solving the original problem (\ref{eq:osc}) is 60.80 seconds
while that of solving 8 subproblems is 9.71 seconds.

\begin{table}[htbp]
    \caption{Statistics of solving the original problem (\ref{eq:osc}) and
    8 subproblems using trilinear finite elements.
    In Column 1, subproblems are labeled by different $\nu$ values.
    In Columns 2 and 3, we list the number of iteration steps and
    matrix-vector multiplications.
    In Columns 4 and 5, we present CPU time spent on
    matrix-vector multiplications and the whole procedure of IRLM.}
    \begin{center} \footnotesize
    \begin{tabular}{|c|cc|cc|}
        \hline
        Problem &\#Iter.&\#OP*x& time\_mv (sec.) & time\_total (sec.) \\
        \hline
        (\ref{eq:osc}) & 22 & 1599  & 175.01  & 930.39  \\ \hline
        $\nu=1$  & 18    & 356  &   5.13  &   8.21  \\
        $\nu=2$  & 22    & 420  &   6.16  &   9.95  \\
        $\nu=3$  & 22    & 421  &   6.06  &   9.83  \\
        $\nu=4$  & 21    & 406  &   5.84  &   9.46  \\
        $\nu=5$  & 18    & 353  &   5.06  &   8.23  \\
        $\nu=6$  & 21    & 405  &   5.74  &   9.44  \\
        $\nu=7$  & 20    & 389  &   5.51  &   9.03  \\
        $\nu=8$  & 21    & 403  &   5.75  &   9.34  \\
        \hline
    \end{tabular}
    \end{center}
    \label{table:cost-fe1}
\end{table}

\begin{table}[htbp]
    \caption{Statistics of solving (\ref{eq:osc}) and 8 subproblems
    using tricubic finite elements.}
    \begin{center} \footnotesize
    \begin{tabular}{|c|cc|cc|}
        \hline
        Problem &\#Iter.&\#OP*x& time\_mv (sec.) & time\_total (sec.) \\
        \hline
        (\ref{eq:osc}) & 57 & 3972 & 1696.29 & 3465.57 \\ \hline
        $\nu=1$  & 50    & 937  &   55.15  &   62.75  \\
        $\nu=2$  & 64    &1156  &   67.75  &   77.11  \\
        $\nu=3$  & 64    &1153  &   67.57  &   77.09  \\
        $\nu=4$  & 62    &1128  &   66.05  &   75.25  \\
        $\nu=5$  & 47    & 892  &   52.18  &   59.31  \\
        $\nu=6$  & 69    &1215  &   71.15  &   81.37  \\
        $\nu=7$  & 63    &1134  &   66.74  &   76.12  \\
        $\nu=8$  & 70    &1230  &   72.46  &   82.87  \\
        \hline
    \end{tabular}
    \end{center}
    \label{table:cost-fe3}
\end{table}

We note that the speedup in average CPU time of a single iteration is 11.71
with trilinear finite elements while it is decreased to 6.26
with tricubic finite elements.
This numerical phenomenon can be explained by performance analysis
(\ref{eq:speedup_rough}). In our computation, the maximum dimension
of Krylov subspace is twice the number of required eigenpairs plus 5,
which is recommended by ARPACK's tutorial examples.
So we have $\theta_1=4.59$.
We obtain from the statistics of solving the original problem that
the CPU time percentage $\omega$ of matrix-vector multiplications is 0.19 with
trilinear finite elements and grows to 0.49 with tricubic finite elements.
Correspondingly, using (\ref{eq:speedup_rough}), we predict that
the CPU time speedup for trilinear and tricubic finite elements would be
12.52 and 7.64, respectively.

We see from (\ref{eq:flops_iteration_parts}) that the computational cost of
$QR$-iteration grows faster than that of matrix-vector multiplication
when the number of required eigenpairs increases.
Thus we can expect that the decomposition approach would be more appreciable
for large-scale eigenvalue problems.

\subsection{Saving in communication} \label{sec:comm_cost}
Besides the reduction in computational cost, solving decoupled problems
will also save communication among parallel processors.
As mentioned in Section~\ref{sec:parallel_impl}, our implementation of
the decomposition approach is parallelized in two levels.
No communication occurs between any two groups of processors
during solving the eigenvalue problem.
This leads to a saving in communication.

For illustration,
we take the oscillator eigenvalue problem (\ref{eq:osc}) as an example.
We decompose it into 8 decoupled subproblems according to group $D_{2h}$.
The comparison of communication between solving the original problem and
the subproblems is given in Table \ref{table:comm}.

\begin{table}[htbp]
    \caption{Comparison of communication between solving (\ref{eq:osc}) and
    8 subproblems.
    Column 1 gives the number of processors.
    In the other columns, ``use symm'' represents solving subproblems
    and ``not use'' means solving the original eigenvalue problem.
    Columns 2 and 3 give the average number of processors
    each processor communicates with.
    Columns 4 and 5 list the average number of Bytes sent by each processor.
    And the last two columns report the CPU time spent on communication
    during matrix-vector multiplications.}
    \begin{center} \footnotesize
    \begin{tabular}{|c|cc|cc|cc|}
        \hline
        \multirow{2}{*}{$N_p$} & \multicolumn{2}{c|}{$N_p$ in comm}
        & \multicolumn{2}{c|}{Bytes in comm}
        & \multicolumn{2}{c|}{CPU time in comm (sec.)} \\
        & use symm&not use  & use symm&not use & use symm&not use  \\\hline
        8   &  0.00 & 1.75 &  0 & 134,560      &  0.00 & 8.93 \\
        16  &  1.00 & 1.88 & 19,608 & 145,451 &  0.20 & 10.36  \\
        \hline
    \end{tabular}
    \end{center}
    \label{table:comm}
\end{table}

\subsection{Applications to Kohn--Sham equations}
Now we apply the decomposition approach to electronic structure calculations
of symmetric molecules, based on code RealSPACES
(Real Space Parallel Adaptive Calculation of Electronic Structure)
of the LSEC of Chinese Academy of Sciences.
In the context of density functional theory (DFT), ground state properties of
molecular systems are usually obtained by solving the Kohn--Sham equation
\cite{hohenberg-kohn64,kohn-sham65,martin04}.
It is a nonlinear eigenvalue problem as follows
\begin{equation}
    \label{eq:kohn-sham}
    \left(-\frac{1}{2}\Delta + V^{\mbox{eff}}[\rho]\right)\Psi_n
    = \epsilon_n\Psi_n \quad\mbox{in}~\mathbb{R}^3,
\end{equation}
where $\rho({\bf r})=\sum_{n=1}^{N_e}f_n\left|\Psi_n({\bf r})\right|^2$
is the charge density contributed by $N_e$ eigenfunctions $\{\Psi_n\}$
with occupancy numbers $\{f_n\}$, and $V^{\mbox{eff}}[\rho]$ the
so-called effective potential which is a nonlinear functional of $\rho$.
On the assumption of no external fields,
$V^{\mbox{eff}}[\rho]$ can be written into
\[ V^{\mbox{eff}} = V^{\mbox{ne}} + V^{\mbox{\small H}} + V^{\mbox{xc}}, \]
where $V^{\mbox{ne}}$ is the Coulomb potential between the nuclei and
the electrons, $V^{\mbox{\small H}}$ the Hartree potential, and
$V^{\mbox{xc}}$ the exchange-correlation potential \cite{martin04}.
The ground state density of a confined system decays exponentially
\cite{agmon81,garding83,simon00}, so we choose the computational domain as
an appropriate cube and impose zero boundary condition.

As a nonlinear eigenvalue problem, Kohn--Sham equation (\ref{eq:kohn-sham})
is solved by the self-consistent field (SCF) iteration \cite{martin04}.
The dominant part of computation is the repeated solving of the
linearized Kohn--Sham equation with a fixed effective potential.
The number of required eigenstates grows in proportion to
the number of valence electrons in the system.
Therefore the Kohn--Sham equation solver will probably
make the performance bottleneck for large-scale DFT calculations.

Real-space discretization methods are attractive for confined systems
since they allow a natural imposition of the zero boundary condition
\cite{beck00,dai11,kronik06}. Among real-space mesh techniques,
the finite element method keeps both locality and the variational property,
and has been successfully applied to electronic structure calculations
(see, e.g.,
\cite{ackermann94,fang12,fattebert07,gong08,pask99,pask05,sterne99,suryanarayana10,tsuchida95,white89,zhang08});
others like the finite difference method, finite volume method and
the wavelet approach have also shown the potential in this field
\cite{chelikowsky94,dai11,genovese08,hasegawa11,iwata10,kronik06,ono05}.

We solve the Kohn--Sham equation of some symmetric molecules
with tricubic finite element discretizations.
The statistics are summarized in Table \ref{table:esc}.
The full symmetry group of these molecules is the tetrahedral group $T_d$.
For simplicity we select subgroup $D_2$ as shown in Table
\ref{table:IR-d2} \cite{cornwell97}.
Accordingly, the Kohn--Sham equation can be decomposed into 4
decoupled subproblems.
It is indicated by the increasing speedup in Table \ref{table:esc} that
the decomposition approach is appreciable for large-scale symmetric
molecular systems.

\begin{table}[htbp]
    \caption{Representation matrices of Abelian group $D_2$.
    It has 4 symmetry operations, i.e., order $g=4$, and thus
    has $n_c=4$ one-dimensional irreducible representations.
    We refer to Appendix A for a description of
    symmetry operations.}
    \begin{center} \footnotesize
    \begin{tabular}{|c|cccc|}
        \hline
        $D_2$ & $E$ & $C_{2x}$ & $C_{2y}$ & $C_{2z}$ \\
        \hline
        $\Gamma^{(1)}$ & 1 & 1 & 1 & 1 \\
        $\Gamma^{(2)}$ & 1 & 1 &-1 &-1 \\
        $\Gamma^{(3)}$ & 1 &-1 & 1 &-1 \\
        $\Gamma^{(4)}$ & 1 &-1 &-1 & 1 \\
        \hline
    \end{tabular}
    \end{center}
    \label{table:IR-d2}
\end{table}

\begin{table}[htbp]
    \caption{Comparison between solving the original Kohn--Sham equation and
    subproblems.
    Column 3 gives the number of required eigenstates.
    The number of degrees of freedom given in Columns 4 and 5 is required by
    the convergence of ground state energy \cite{fang12}.
    Columns 7 and 8 list the average CPU time in diagonalization
    at each SCF iteration step, which is the dominant part of time.
    The last column is the speedup of the decomposition approach.}
    \begin{center} \footnotesize
    \begin{tabular}{|c|c|c|cc|c|cc|c|}
        \hline
        \multirow{2}{*}{System} & \multirow{2}{*}{$G$} &
        \multirow{2}{*}{$N_e$} & \multirow{2}{*}{$N$}&\multirow{2}{*}{$N_0$}&
        \multirow{2}{*}{$N_p$} & \multicolumn{2}{c|}{CPU time in diag. (sec.)}
        & \multirow{2}{*}{Speedup}
        \\
        & & &  & & & not use&use symm & \\
        \hline
        $\mbox{C}_{123}\mbox{H}_{100}$ & $D_2$
        & 300 &  1,191,016 & 297,754 & 32 & 2,783 & 558 & 4.99 \\
        \hline
        $\mbox{C}_{275}\mbox{H}_{172}$ & $D_2$
        & 640  & 1,643,032 & 410,758 & 32 & 13,851 & 1,559 & 8.88 \\
        \hline
        $\mbox{C}_{525}\mbox{H}_{276}$ & $D_2$
        & 1200 & 2,097,152 & 524,288 & 64 & 25,296 & 2,334 & 10.84 \\
        \hline
    \end{tabular}
    \end{center}
    \label{table:esc}
\end{table}

\section{Concluding remarks}
In this paper, we have proposed a decomposition approach to
eigenvalue problems with spatial symmetries.
We have formulated a set of eigenvalue subproblems friendly for
grid-based discretizations.
Different from the classical treatment of symmetries in quantum chemistry,
our approach does not explicitly construct symmetry-adapted bases.
However, we have provided a construction procedure for
the symmetry-adapted bases, from which we have obtained the relation between
the two approaches.

Note that such a decomposition approach can reduce the computational cost
remarkably since only a smaller number of eigenpairs are solved for
each subproblem and the subproblems can be solved in a smaller subdomain.
We would believe that the quantization of this reduction implies that
our approach could be appreciable for large-scale eigenvalue problems.
In practice, we solve a sufficient number of redundant eigenpairs
for each subproblem in order not to miss any eigenpairs.
It would be very helpful for reducing the extra work if one could
predict the distribution of eigenpairs among subproblems.

Under finite element discretizations, our decomposition approach has been
applied to Kohn--Sham equations of symmetric molecules.
If solving Kohn--Sham equations of periodic crystals,
we should consider plane wave expansion which could be regarded as
grid-based discretization in reciprocal space.
In Appendix C, we show that the invariance
under some coordinate transformation can be kept by Fourier transformation.
So the decomposition approach would be applicable to plane waves, too.

Currently, we have imposed an odd number of partition and used finite elements
of odd orders to avoid degrees of freedom on symmetry elements.
In numerical examples, we have treated only a part of cubic symmetries
for validation and illustration.
Obviously, the decomposition approach and its practical issues
can be adapted to other spatial symmetries with appropriate grids.

In this paper, we concentrate on spatial symmetries only.
It is possible to use other symmetries to reduce the computational cost, too.
For instance, the angular momentum, spin and parity symmetries of atoms
have been exploited during solving the Schr{\"o}dinger equation in
\cite{friesecke09,mendl10};
the total particle number and the total spin $z$-component,
except for rotational and translational symmetries,
have been taken into account to block-diagonalize
the local (impurity) Hamiltonian in the computation of
dynamical mean-field theory for strongly correlated systems
\cite{gull11,haule07}.
It is our future work to exploit these underlying or internal symmetries.

\section*{Appendix A: Basic concept of group theory}
\renewcommand{\theequation}{A.\arabic{equation}}
\renewcommand{\thesection}{A}
\setcounter{equation}{0} 
In this appendix, we include some basic concepts of group theory
for a more self-contained exposition.
They could be found in standard textbooks like
\cite{bishop93,cornwell97,cotton90,jones60,tinkham64,wigner59}.

A group $G$ is a set of elements $\{R\}$ with a well-defined multiplication
operation which satisfy several requirements:
\begin{enumerate}
    \item The set is closed under the multiplication.
    \item The associative law holds.
    \item There exists a unit element $E$ such that $ER=RE=R$ for any $R\in G$.
    \item There is an inverse $R^{-1}$ in $G$ to each element $R$
        such that $RR^{-1} = R^{-1}R = E$.
\end{enumerate}

If the commutative law of multiplication also holds,
$G$ is called an Abelian group.
Group $G$ is called a finite group if it contains a finite number of elements.
And this number, denoted by $g$, is said to be the order of the group.
The rearrangement theorem tells that the elements of $G$ are only rearranged
by multiplying each by any $R\in G$, i.e., $RG=G$ for any $R\in G$.

An element $R_1\in G$ is called to be conjugate to $R_2$ if $R_2=SR_1S^{-1}$,
where $S$ is some element in the group.
All the mutually conjugate elements form a class of elements.
It can be proved that group $G$ can be divided into distinct classes.
Denote the number of classes as $n_c$.
In an Abelian group, any two elements are commutative, so each element
forms a class by itself, and $n_c$ equals the order of the group.

Two groups is called to be homomorphic if there exists a correspondence
between the elements of the two groups as $R\leftrightarrow R'_1,R'_2,\ldots$,
which means that if $RS=T$ then the product of any $R'_i$ with any $S'_j$
will be a member of the set $\{T'_1,T'_2,\ldots\}$.
In general, a homomorphism is a many-to-one correspondence. It specializes to
an isomorphism if the correspondence is one-to-one.

A representation of a group is any group of mathematical entities
which is homomorphic to the original group.
We restrict the discussion to matrix representations.
Any matrices representation with nonvanishing determinants is equivalent to
a representation by unitary matrices.
Two representations are said to be equivalent if they are associated by
a similarity transformation.
If a representation can not be equivalent to representations of lower
dimensionality, it is called irreducible.

The number of all the inequivalent, irreducible, unitary representations
is equal to $n_c$, which is the number of classes in $G$.
The Celebrated Theorem tells that
\[ \sum_{\nu=1}^{n_c}d_{\nu}^2=g, \]
where $d_{\nu}$ denotes the dimensionality of the $\nu$-th representation.
Since the number of classes of an Abelian group equals the number of elements,
an Abelian group of order $g$ has $g$ one-dimensional irreducible
representations.

The groups used in this paper are all crystallographic point groups.
Groups $D_2$, $D_{2h}$, $D_{2d}$ and $D_4$
are four dihedral groups; the first two groups are Abelian and the other two
are non-Abelian.
In Table \ref{table:IR} and Table \ref{table:IR-d2}, $C_{nj}$ denotes
a rotation about axis $Oj$ by $2\pi/n$ in the right-hand screw sense and
$I$ is the inversion operation \cite{cornwell97}.
The $Oj$ axes are illustrated in Figure~\ref{fig:axes}.
We refer to textbooks like \cite{bishop93,cornwell97,cotton90,tinkham64}
for more details about crystallographic point groups.
\begin{figure}[htpb]
    \begin{center}
        \includegraphics[width=0.35\textwidth]{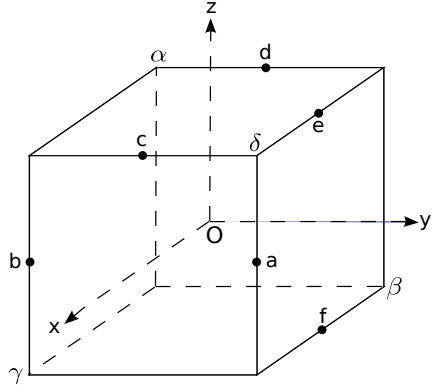}
    \end{center}
    \caption{Rotation axes in Table \ref{table:IR} and Table \ref{table:IR-d2}}
    \label{fig:axes}
\end{figure}

\section*{Appendix B: Proof of Proposition \ref{prop:Pml_property}}
\renewcommand{\theequation}{B.\arabic{equation}}
\renewcommand{\thesection}{B}
\setcounter{equation}{0} 
\begin{proof}
    (a) Since $\{P_R\}$ are unitary operators, we have
    \[ {\mathscr{P}^{(\nu)}_{ml}}^* = \left(\frac{d_{\nu}}{g} \sum_{R\in G}
    {\Gamma^{(\nu)}(R)}^*_{ml} P_R \right)^* = \frac{d_{\nu}}{g} \sum_{R\in G}
    \Gamma^{(\nu)}(R)_{ml} P_{R^{-1}} = \frac{d_{\nu}}{g} \sum_{S\in G}
    \Gamma^{(\nu)}(S^{-1})_{ml} P_S, \]
    which together with the fact that $\Gamma^{(\nu)}$ is a unitary
    representation derives
    \[ {\mathscr{P}^{(\nu)}_{ml}}^* = \frac{d_{\nu}}{g} \sum_{S\in G}
    {\Gamma^{(\nu)}(S)}^*_{lm} P_S  = \mathscr{P}^{(\nu)}_{lm}. \]

    (b) It follows from the definition that
    \begin{eqnarray*}
        \mathscr{P}^{(\nu)}_{ml} \mathscr{P}^{(\nu')}_{m'l'}
        &=& \left(\frac{d_{\nu}}{g}\sum_{R\in G}
        {\Gamma^{(\nu)}(R)}^*_{ml}P_R\right)
        \left(\frac{d_{\nu'}}{g} \sum_{S\in G}
        {\Gamma^{(\nu')}(S)}^*_{m'l'}P_S\right)
        \\
        &=& \frac{d_{\nu}d_{\nu'}}{g^2} \sum_{R\in G}{\Gamma^{(\nu)}(R)}^*_{ml}
        \left( \sum_{S\in G} {\Gamma^{(\nu')}(S)}^*_{m'l'}P_{RS} \right).
    \end{eqnarray*}
    Note that the rearrangement theorem implies that,
    when $S$ runs over all the group elements, $S'=RS$ for any $R$
    also runs over all the elements. Hence we get
    \begin{eqnarray*}
        \mathscr{P}^{(\nu)}_{ml} \mathscr{P}^{(\nu')}_{m'l'}
        &=& \frac{d_{\nu}d_{\nu'}}{g^2} \sum_{R\in G}{\Gamma^{(\nu)}(R)}^*_{ml}
        \left( \sum_{S'\in G} {\Gamma^{(\nu')}(R^{-1}S')}^*_{m'l'}P_{S'}
        \right)\\
        &=& \frac{d_{\nu}d_{\nu'}}{g^2} \sum_{S'\in G} \left(\sum_{R\in G}
        {\Gamma^{(\nu)}(R)}^*_{ml}~{\Gamma^{(\nu')}(R^{-1}S')}^*_{m'l'}
        \right) P_{S'}.
    \end{eqnarray*}
    We may calculate as follows
    \begin{eqnarray*}
        \sum_{R\in G}
        {\Gamma^{(\nu)}(R)}^*_{ml} {\Gamma^{(\nu')}(R^{-1}S')}^*_{m'l'}
        &=& \sum_{R\in G} {\Gamma^{(\nu)}(R)}^*_{ml}
        \left( \sum_{n=1}^{d_{\nu'}}
        {\Gamma^{(\nu')}(R^{-1})}^*_{m'n}~{\Gamma^{(\nu')}(S')}^*_{nl'}
        \right) \\
        &=& \sum_{R\in G} {\Gamma^{(\nu)}(R)}^*_{ml}
        \left( \sum_{n=1}^{d_{\nu'}}
        {\Gamma^{(\nu')}(R)}_{nm'}~{\Gamma^{(\nu')}(S')}^*_{nl'}
        \right) \\
        &=& \sum_{n=1}^{d_{\nu'}} {\Gamma^{(\nu')}(S')}^*_{nl'} \left(
        \sum_{R\in G} {\Gamma^{(\nu)}(R)}^*_{ml}~{\Gamma^{(\nu')}(R)}_{nm'}
        \right),
    \end{eqnarray*}
    which together with the great orthogonality theorem yields
    \[ \sum_{R\in G}
    {\Gamma^{(\nu)}(R)}^*_{ml} {\Gamma^{(\nu')}(R^{-1}S')}^*_{m'l'}
    = \delta_{\nu\nu'} \delta_{lm'}\frac{g}{d_{\nu'}}
    {\Gamma^{(\nu)}(S')}^*_{ml'}. \]
    Thus we arrive at
    \[ \mathscr{P}^{(\nu)}_{ml} \mathscr{P}^{(\nu')}_{m'l'} =
    \delta_{\nu\nu'} \delta_{lm'}
    \frac{d_{\nu}}{g} \sum_{S'\in G} {\Gamma^{(\nu)}(S')}^*_{ml'} P_{S'}
    = \delta_{\nu\nu'} \delta_{lm'} \mathscr{P}^{(\nu)}_{ml'}. \]
\qquad\end{proof}

\section*{Appendix C: Spatial symmetry in reciprocal space}
\renewcommand{\theequation}{C.\arabic{equation}}
\renewcommand{\thesection}{C}
\setcounter{equation}{0}
Plane wave method is widely used for solving the Kohn--Sham equations
of crystals. Actually, plane waves may be regarded as grid-based discretizations
in reciprocal space. We will show that the symmetry relation in real space
is kept in reciprocal space.
The solution domain $\Omega$ of crystals can be spanned by
three lattice vectors in real space. We denote them as ${\bf a}_i (i=1,2,3)$.
If function $f$ is invariant with integer multiple translations of
the lattice vectors, we then present the function in reciprocal space as like:
\[ \hat{f}({\bf q})=\frac{1}{N}\sum_{ \bf{r}}f({\bf r})
e^{-\imath {\bf q}\cdot{\bf r}}, \]
where ${\bf q}$ is any vector in reciprocal space satisfying
${\bf q}\cdot{\bf a}_i=2\pi\frac{n}{N_i}$ with $n$ an integer,
$N_i$ the number of degrees of freedom along direction ${\bf a}_i~(i=1,2,3)$,
and $N=N_1N_2N_3$ the total number of degrees of freedom.
Assume that $f$ is kept invariant under coordinate transformation $R$
in $\Omega$. We obtain from
\[ \hat{f}(R{\bf q}) = \frac{1}{N}\sum_{ {\bf r}}f({\bf r})
e^{-\imath (R{\bf q})\cdot{\bf r}} \]
and the coordinate transformation $R$ can be represented as
an orthogonal matrix that
\[ \hat{f}(R{\bf q}) = \frac{1}{N}\sum_{ {\bf r}}f( {\bf r})
e^{-\imath {\bf q}\cdot\left(R^{-1}{\bf r}\right)}. \]
Since
\[f(R^{-1}{\bf r})=f({\bf r})\quad\forall {\bf r}\in\Omega,\]
we have
\[ \hat{f}(R{\bf q}) = \frac{1}{N}\sum_{R^{-1}{\bf r}}f(R^{-1}{\bf r})
e^{-\imath {\bf q}\cdot\left(R^{-1}{\bf r}\right)} = \hat{f}({\bf q}). \]
Hence the decomposition approach is probably applicable to plane waves.

\vskip 0.2cm
{\sc Acknowledgements.} 
The authors would like to thank Prof. Xiaoying Dai, Prof. Xingao Gong,
Prof. Lihua Shen, Dr. Zhang Yang, and Mr. Jinwei Zhu for
their stimulating discussions on electronic structure calculations.
The second author is grateful to Prof. Zeyao Mo for his encouragement.


\end{document}